\documentclass[10pt]{article}
\usepackage{amsmath, amssymb, amsfonts, amsxtra, latexsym, amscd, eucal, amsthm, graphicx, epsfig, color, enumerate}  
\usepackage{array} 
\usepackage{hyperref}   
\theoremstyle{theorem}
\newtheorem{Def}{Definition}[section]

\newtheorem{Prop}[Def]{Proposition}
\newtheorem{Lem}[Def]{Lemma}
\newtheorem{Thm}[Def]{Theorem}

\theoremstyle{definition}
\newtheorem{Rem}[Def]{Remark}

\newcommand{\p}{\mathbb{P}}

\newcommand{\bR}{\mathbb{R}}

\newcommand{\mf}{\mathcal{F}}

\newcommand{\pr}{\mathbb{P}}

\newcommand{\bE}{\mathbb{E}}
 
\newcommand{\br}{\mathbb{R}} 

\newcommand{\sm}{\sigma}

\newcommand{\Om}{\Omega}

\usepackage[top=2cm, bottom=2cm, left=2cm, right=2cm]{geometry}

\allowdisplaybreaks

\begin{document}

\title{A tamed-adaptive Milstein scheme for stochastic differential equations with  low regularity coefficients}
\author{ Thi-Huong Vu \footnote{University of Transport and Communications, 3 Cau Giay, Lang Thuong, Dong Da, Hanoi, Vietnam. Email: vthuong@utc.edu.vn} \footnote{University of Science, Vietnam National University, Hanoi, 334 Thanh Xuan, Hanoi, Vietnam}  \quad  Hoang-Long Ngo\footnote{Hanoi National University of Education, 136 Xuan Thuy, Cau Giay, Hanoi, Vietnam. Email: ngolong@hnue.edu.vn}\quad  Duc-Trong Luong\footnote{Hanoi National University of Education, 136 Xuan Thuy, Cau Giay, Hanoi, Vietnam. Email: trongld@hnue.edu.vn} \quad Ngoc Khue Tran \footnote{Corresponding author. Faculty of Mathematics and Informatics, Hanoi University of Science and Technology, 1 Dai Co Viet, Bach Mai, Hanoi, Vietnam. Email: khue.tranngoc@hust.edu.vn} }

\maketitle
\begin{abstract} 
We propose a tamed-adaptive Milstein scheme for stochastic differential equations in which the first-order derivatives of the coefficients are locally H\"older continuous of order $\alpha$. We show that the scheme converges in the $L_2$-norm with a rate of $(1+\alpha)/2$ over both finite intervals $[0, T]$ and the infinite interval $(0, +\infty)$, under certain growth conditions on the coefficients.
\end{abstract} 

\section{Introduction}
This paper considers the numerical approximation for a one-dimensional process $X=(X_t)_{t \geq 0}$ which is a solution to the following stochastic differential equation (SDE)
\begin{equation} \label{eqn1}
X_t=x_0+\int_0^t \mu(X_s)ds+\int_0^t \sigma(X_s)dW_s, \quad x_0\in \mathbb{R}, \quad t \geq 0,
\end{equation}
where $W=(W_t)_{t\geq 0}$ is a one-dimensional standard Brownian motion defined on a complete filtered probability space $(\Om, \mf, (\mf_t), \pr )$. 

Stochastic differential equations are widely employed in modeling systems across various fields, including biology, physics, and economics \cite{I1, I3, I2}. Due to the limited availability of explicit solutions, numerical approximation methods have become indispensable for many applications. The convergence rates of these numerical schemes typically depend on the growth conditions and smoothness of the coefficients. 
It is well-known that for SDEs whose coefficients $\mu$ and $\sigma$ satisfy global Lipschitz continuity and linear growth conditions, the Euler–Maruyama method attains a strong convergence rate of order $1/2$. Furthermore, if  $\mu$ is once differentiable and $\sigma$ is twice differentiable, then the Milstein scheme  attains a higher strong convergence rate of order $1$ (see \cite{I31, I4, I5}). 

However, recent studies \cite{I6} revealed that for SDEs with superlinearly growing coefficients, classical explicit schemes such as the Euler-Maruyama and Milstein schemes may fail to converge in the $L_p$-norm. This limitation has motivated the development of various alternative numerical methods tailored for SDEs with non-globally Lipschitz coefficients. The tamed Euler-Maruyama scheme was introduced in \cite{I19} and further developed in \cite{Sabanis1, Sabanis2, NL19}. Another notable approach is the truncated Euler scheme introduced in \cite{Mao}. These modified schemes are designed to handle superlinear growth and have been shown to achieve a strong convergence rate of order 
$1/2$ for certain classes of SDEs, under appropriate one-sided Lipschitz and monotonicity conditions.
  
Several implicit Milstein-type schemes have been developed for nonlinear SDEs, which achieve strong convergence of order $1$; see, for example, \cite{I10, I12, I13}. However, it is important to emphasize that implicit methods generally entail significant computational cost, as they require solving nonlinear algebraic equations at each time step.
To overcome this limitation, explicit variants have been proposed. Notably, the tamed Milstein scheme introduced in \cite{I20} attains strong convergence of order $1$ for a class of SDEs with superlinearly growing coefficients. Besides, the truncated Milstein scheme, initially investigated in \cite{Mao} for SDEs with commutative noise, was shown in \cite{I16} to achieve the same convergence rate under the assumption that both the drift and diffusion coefficients possess continuous second-order derivatives and are polynomially bounded.
Further refinements and generalizations of the tamed and truncated Milstein schemes can be found in \cite{I21, I16, L22, I23, I26, I17, I25, I24, Jiang}, and references therein.

An alternative to fixed-step numerical methods is the class of adaptive schemes, in which each time step is dynamically determined based on the current state of the approximated process. Fang and Giles \cite{I27} proposed an adaptive Euler–Maruyama scheme tailored for SDEs with polynomially growing drift, satisfying a one-sided Lipschitz condition, and with bounded, globally Lipschitz continuous diffusion coefficients.
Building on this framework, a tamed-adaptive Euler-Maruyama scheme was developed in \cite{KLN, KNLT, KNLT25} by integrating the taming technique into the adaptive time-stepping strategy, thereby extending the applicability of explicit methods to a broader class of SDEs with superlinear growth.

The tamed-adaptive Euler-Maruyama scheme is applicable to SDEs whose drift coefficients are both locally Lipschitz continuous and one-sided Lipschitz continuous, and whose diffusion coefficients are locally 
$(\alpha + 1/2)$-H\"older continuous. In \cite{I28}, the authors introduced a class of path-dependent time-stepping strategies, wherein the step size is adaptively reduced when the numerical solution approaches the boundary of a prescribed sphere. This approach was employed to construct an explicit adaptive Milstein scheme applicable to SDEs with non-commutative noise, i.e., where the commutativity condition is not satisfied. It was demonstrated that this scheme achieves strong convergence of order $1$ under the assumption that the drift coefficient satisfies a one-sided Lipschitz condition.


In this paper, we propose a tamed-adaptive Milstein scheme for equation \eqref{eqn1}, which synthesizes key elements from both the tamed Milstein method and adaptive time-stepping strategies. \textcolor{black}{Notably, due to the adaptive control of the step size, it is not necessary to tame the drift coefficient $b$ or the diffusion coefficient $\sigma$. Instead, only the composite term $\sigma \sigma'$  requires taming. This distinction sets our approach apart from the methods introduced in \cite{KLN, KNLT, KNLT25}, highlighting a fundamental difference in the treatment of nonlinearities.}

The novel contribution of this work lies in the significant relaxation of the regularity assumptions imposed on the coefficients of the SDE. In contrast to prior studies that typically require global Lipschitz continuity, we assume only local $\alpha$-H\"older continuity of the first derivatives of the drift and diffusion coefficients, $\mu^{\prime}$ and $\sigma^{\prime}$.
Under these relaxed conditions, we prove that the proposed tamed-adaptive Milstein scheme  strongly converges in the $L_2$-norm with a rate of $(1+\alpha) / 2$, thereby quantifying the influence of the regularity of $\mu^{\prime}$ and $\sigma^{\prime}$ on the convergence behavior of the Milstein method. Notably, when $\alpha=1$, the strong convergence rate reaches the optimal order of $1$.
Moreover, we conduct a detailed analysis of the strong convergence rate on both finite time intervals $[0, T]$, for fixed $T>0$, and the infinite time horizon $[0, \infty)$. The study of numerical schemes over infinite time intervals has only recently attracted attention in the literature (see \cite{I27, KLN, KNLT, KNLT25}). In this context, the integrability properties of both the exact and approximate solutions play a critical role in determining the attainable convergence rate.

The remainder of this paper is structured as follows. In Section \ref{sec:2}, we outline the assumptions and main results of the model. Specifically, we define the tamed-adaptive Milstein scheme and present key theorems regarding moment estimates for the approximate process and the strong convergence of the scheme. All proofs are deferred to Section \ref{sec:3}. In Section \ref{sec:nume}, we provide a numerical analysis of the tamed-adaptive Milstein scheme to illustrate the theoretical findings. \textcolor{black}{We also compare the performance of the tamed-adaptive Milstein scheme with that of the tamed Milstein scheme proposed in \cite{I23}.} Technical estimates for the increments of both exact and approximate solutions are provided in Section \ref{sec:app}.


\section{Model assumptions and main results} 
\label{sec:2}
For each $\alpha \in (0,1]$ and $l \geq 0$,  let $C^{1+\alpha}_l$ denote the class of differentiable functions $f: \mathbb{R} \to \mathbb{R}$ such that 
$$|f'(x) - f'(y)| \leq L |x-y|^\alpha (1 + |x|^{l} + |y|^{l}), \text{ for all } x, y \in \mathbb{R},$$
for some positive constant $L$. 

Note that if $0< \alpha < \alpha' \leq 1$ then $C^{1+\alpha'}_l \subset C^{1+\alpha}_{l + \alpha' - \alpha}$, and if $0 \leq l < l'$ then $C^{1+\alpha}_{l} \subset C^{1+\alpha}_{l'}$.

Throughout the paper, we suppose that the coefficients $\mu$ and $\sigma$ of the equation \eqref{eqn1} satisfy the following assumptions. 
\begin{itemize}
\item[\bf A0.] $\mu, \sigma \in C^{1+\alpha}_l$ for some $\alpha \in (0,1]$ and $ l \geq 0$.
\item[\bf A1.]  There exist constants  $\gamma \in \mathbb R$,  $\eta \in [0, +\infty)$ and $p_0 \in [4(l + \alpha+1), +\infty)$ such that  for any $x \in \br$, it holds
	$$x\mu(x) + \dfrac{p_0-1}{2}  \sm^2(x)  \leq \gamma  x^2+\eta.$$
\item[\bf A2.] There exists a constant $\lambda \in \mathbb{R}$ such that for any $ x, y \in \mathbb{R}$, it holds 
 $$ (x-y)(\mu(x)- \mu(y))+ \dfrac{1}{2}|\sm(x)-\sm(y)|^2 \leq \lambda|x-y|^2.$$
 \end{itemize}
Assumptions \textbf{A0} and \textbf{A1} guarantee the existence of a unique, adapted process $X = (X_t)_{t \geq 0}$ that solves equation \eqref{eqn1} (see \cite[Theorem 2.1]{NL19}). We now  state  the moment estimates for the exact solution, which are drawn from Proposition 2.3 in \cite{KLN}.
 \begin{Prop}\label{moment nghiem dung}\cite[Proposition 2.3]{KLN}
	Let $X=(X_t)_{t \geq 0}$ be a solution to equation \eqref{eqn1}. Then, for any $p \in (0,p_0]$ and $t\geq 0$,
	\begin{align*} 
		\bE \left[\vert X_t\vert^p\right] \le \begin{cases} \left \vert x_0^2 e^{2\gamma t}+\frac{\eta}{\gamma}(e^{2\gamma t}-1)\right\vert ^{p/2} & \text{ if \;} \gamma \not = 0,\\
			\left\vert x_0^2+ 2\eta t \right\vert ^{p/2} & \text{ if\; } \gamma  = 0. \end{cases} 
	\end{align*}
\end{Prop}
\begin{Rem}\label{rm1} When $\gamma<0$, there exists a positive constant $C = C_p$  such that  $$\sup_{t \geq 0} \bE \left[\vert X_t\vert^p\right] \le C_p.$$
\end{Rem}


For each $\Delta \in (0,1)$, the tamed-adaptive Milstein approximation scheme of equation \eqref{eqn1}  is defined by
\begin{equation}\label{Milstein1}
\begin{cases} 
t_0=0, \quad \widehat{X}_0=x_0, \quad t_{k+1}=t_k+h(\widehat{X}_{t_k})\Delta,\\ 
  \widehat{X}_{t_{k+1}}= \widehat{X}_{t_k}+\mu(\widehat{X}_{t_k})\left(t_{k+1}-t_k\right)+\sm(\widehat{X}_{t_k})(W_{t_{k+1}}-W_{t_{k}}) + \dfrac{1}{2}q_{\Delta}(\widehat{X}_{t_k})\left((W_{t_{k+1}}-W_{t_k})^2 -(t_{k+1}-t_k)\right),
\end{cases} 
\end{equation}
where $$q_{\Delta}(x)= \dfrac{\sm(x)\sm'(x)}{1+\Delta^{1/2}|\sm(x)\sm'(x)|},$$
and 
\begin{equation} \label{chooseh} 
h(x)=\dfrac{h_0}{(1+|\mu(x)|^2+|\mu'(x)|+|\sm(x)|^4+|\sm'(x)|^4+|q_\Delta (x)|+|x|^{l_0})^2},
\end{equation}
 for any $x \in \br$ and some constants $l_0 \ge  2$, $h_0>0$. 

\textcolor{black}{The adaptive Euler-Maruyama schemes in \cite{I27, KLN, KNLT, KNLT25} were designed to control the growth rate of the solution when $b$ and $\sigma$ take large values.  The term $\sigma \sigma'$, which appears in the classical Milstein scheme, can grow even faster than $b$ and $\sigma$. 
Therefore, in the above scheme, we had to apply taming to this term to ensure the integrability of the approximate solution.}

Now, we provide conditions to guarantee that $t_k \uparrow \infty$ as $k \uparrow \infty$, which shows that the tamed-adaptive Milstein approximation scheme \eqref{Milstein1} is well-defined. 

\begin{Prop}\label{Prop3.1} It holds 
	\begin{align} \label{tkToInf}
	\lim\limits_{k \to +\infty} t_k =+\infty \quad \text{a.s.} 
	\end{align}
\end{Prop}

The proof of Proposition \ref{Prop3.1} is deferred to Section \ref{proof tk}.

For any $t\geq 0$, we define the closest time point before $t$ by $\underline{t} := \max \left\{t_{k}: t_{k} \leq t\right\}$.
 Note that $\underline{t}$ is a stopping time. Consequently, we define the standard continuous interpolant by
\begin{equation}\label{Milstein3}
\widehat{X}_{t}=\widehat{X}_{\underline{t}}+\mu(\widehat{X}_{\underline{t}})(t-\underline{t})+\sigma(\widehat{X}_{\underline{t}})\left(W_{t}-W_{\underline{t}}\right)+\dfrac{1}{2}q_{\Delta}(\widehat{X}_{\underline{t}})\left(\left(W_{t}-W_{\underline{t}}\right)^2- (t- \underline{t})\right).
\end{equation}
Thus, $\widehat{X}=(\widehat{X}_t)_{t \geq 0}$ satisfies the following SDE
\begin{equation*}
d \widehat{X}_{t}= \mu (\widehat{X}_{\underline{t}}) dt+\sm(\widehat{X}_{\underline{t}})  dW_{t} +q_{\Delta}( \widehat{X}_{\underline{t}})\left(W_{t}-W_{\underline{t}}\right)  dW_{t}, \qquad \widehat{X}_{0} = x_0.
\end{equation*}

Next, we show the following crucial result on the moment estimates of the modified Milstein approximation.
\begin{Prop}\label{mmX^}
	 For any  $ l_0 \geq 2$,  and $k 
	\leq p_0/2,$  there exists a positive constant $C$ which does not depend on $\Delta$ such that,  for any $t\geq 0$,
	\begin{align}\label{EXtmu2}
	\bE \left[|\widehat{X}_t|^{2k} \right] \vee \bE \left[|\widehat{X}_{\underline{t}}|^{2k}\right] \le \left\{ \begin{array}{l l}
	Ce^{{2k\gamma t}} \quad &\text{ if\; \;} \gamma  > 0,\\
	C(1+t)^{k} \quad &\text{ if\; \;} \gamma =0,\\
	C\quad &\text{ if\; \;} \gamma <0.
	\end{array} \right. 
	\end{align}
	\end{Prop}
	The proof of Proposition \ref{mmX^} is given in Section \ref{proof moment}.
	\begin{Rem}\label{rm2}

	If $\gamma<0$, then for any $ 0\leq p \leq 2 \textcolor{black}{\lfloor p_0/2\rfloor}$ where \textcolor{black}{$\lfloor\;\rfloor$} denotes the integer part,  there exists a positive constant $C=C_p$ which does not depend on $\Delta$ such that 
	$$\sup_{t \geq 0}  \left(\bE \left[|\widehat{X}_t|^p\right] \vee \bE \left[|\widehat{X}_{\underline{t}}|^p\right] \right) \leq C.$$
	\end{Rem}

Let $N(T)$ be the number of timesteps required for a path approximation on $[0, T]$ for any $T > 0$. We have the following estimate on the expectation of $N(T)$.
\begin{Prop} \label{ulN} 
There exists a positive constant $C$, which does not depend on $
\Delta$ or $T$, such that 
$$\mathbb{E}\left[\left|N(T)-1\right|\right] \leq \frac{CT}{\Delta} \sup_{0\leq s \leq T}\bE\left[ 1+ |\widehat{X}_{\underline{s}}|^{2l_0} + |\widehat{X}_{\underline{s}}|^{4(l + \alpha+1)}\right].$$
\end{Prop}

Now, we state the main result which establishes the strong convergence rate of the tamed-adaptive Milstein scheme. 

\begin{Thm}\label{main Thr } 
	Assume conditions \textbf{A0-A2} hold. 
	   Then, for any $T>0$,  $ l_0 \geq \max\{2,  \frac{4l}{3(1+\alpha)}\},$ $p_0 \ge 4(l + \alpha + 1)$, there exists a positive constant $C_T$ depending on $x_0, \gamma, \eta, \alpha, l, l_0, \lambda, p_0$, and  $T$ such that
	\begin{equation} \label{r1}
	\underset{0 \le t \le T}{\sup} \bE \left[  |\widehat{X}_t - X_t| ^2\right] \le 
	C_T \Delta^{1+\alpha}.
	\end{equation}
	Moreover, assume that $\lambda<0$ and $\gamma <0$, then there exists a positive constant $C$, which depends on $x_0, \gamma, \eta, \alpha, l, l_0, \lambda, p_0$, but is independent of $T$,  such that
	\begin{equation} \label{r2}
	\underset{t \ge 0}{\sup} \ \bE \left[  |\widehat{X}_t - X_t| ^2\right] \le C\Delta^{1+\alpha}.
	\end{equation} 
\end{Thm}

The proof of Theorem \ref{main Thr } is deferred to Section \ref{proof3}.

\section{Proof of main results}
\label{sec:3}
Thanks to \textbf{A0} and the definition of $q_\Delta(x)$,  there exists a positive constant $L$ such that for any $x, y \in \mathbb{R}$, the following estimates hold 
\begin{align} 
& |\sm(x)- \sm(y)| \vee |\mu(x)- \mu(y)|  \leq L(1+|x|^{l + \alpha}+|y|^{l + \alpha})|x-y|; \label{U2}\\ 
& |\mu(x)| \vee |\sm(x)| \leq L(1+|x|^{l + \alpha + 1}); \label{U4}\\ 
& |\sm'(x)- \sm'(y)| \vee |\mu'(x)- \mu'(y)| \leq L(1+|x|^{l}+|y|^{l})|x-y|^\alpha; \label{U6}\\
& |\sm(x)\sm'(x)- q_{\Delta}(x)| \leq L \Delta^{1/2}|\sm(x) \sm'(x)|^2; \label{U8}\\
& |q_{\Delta}(x)| \leq \frac{1}{\Delta^{1/2}}. \label{C2}
\end{align}

Note that if $f \in C^{1+\alpha}_l$, then, for any $x, y \in \mathbb{R}$,
	\begin{align}\label{C^1+alpha}
	|f(x)-f(y)-f'(y)(x-y)|& = \left|\int_0^1\left[f'(y+\theta(x-y))-f'(y)\right](x-y)d\theta \right|
	\notag\\
	&\leq \left|\int_0^1L\left(1+|y+\theta(x-y)|^{l}+|y|^{l}\right)|\theta(x-y)|^{\alpha}(x-y)d\theta \right| \notag\\
	&\leq \textcolor{black}{C} \left(1+|x|^{l}+|y|^{l}\right)|x-y|^{1+\alpha},
	\end{align}
for some positive constant $\textcolor{black}{C}$.

\textcolor{black}{In the proofs of the following results, we will use $C_1, C_2, \dots$ to denote positive constants that do not depend on $H$, $\Delta$ or the time variables $s, t, T$. Moreover, each $C_i$, $i\geq 1$, may take different values in the proof of each lemma, proposition or theorem}.

\subsection{Proof of Proposition \ref{Prop3.1} } \label{proof tk}
In \cite{I27, KLN, KNLT}, the authors construct a $K$-scheme, $\hat{X}^K$, based on a projection operator, to establish \eqref{tkToInf}. However, since the projected process $\hat{X}^K$ is not an It\^o process, analyzing it using  It\^o's formula becomes challenging. This complexity is further amplified when applying such an approach to our Milstein scheme.

In response, we consider an alternative method to demonstrate \eqref{tkToInf}. Specifically, following the idea in \cite{KNLT25}, we construct an auxiliary process, $\hat{X}^K$, by truncating the coefficients $\mu$ and $\sigma$ so that $\hat{X}^K$ remains an It\^o process. To address the superlinear growth of $\sigma$, we employ the Yamada-Watanabe approximation (see \cite{GR, Yamada}), which is defined as follows.
For each $\delta>1$ and $\varepsilon>0$, there exists a continuous function $\psi_{\delta\varepsilon}:\mathbb{R}\to\mathbb{R}_+$ with supp$\psi_{\delta\varepsilon}\subset[\frac{\varepsilon}{\delta},\varepsilon]$ such that
\begin{equation*}
\int_{\frac{\varepsilon}{\delta}}^\varepsilon \psi_{\delta\varepsilon}(z)dz=1, \quad 0\leq \psi_{\delta\varepsilon}(z)\leq \dfrac{2}{z\log \delta},\quad z>0.
\end{equation*}
Then, we define
\begin{equation*}
\phi_{\delta\varepsilon}(x):=\int_0^{|x|}\int_0^y\psi_{\delta\varepsilon} (z) dzdy,\quad x\in\mathbb{R}.
\end{equation*}
 The function $\phi_{\delta \epsilon}$ possesses the following  properties: for any $x\in\mathbb{R}\setminus \{0\}$,
\begin{enumerate}[\bf  YW1.]
	\item $\phi'_{\delta\varepsilon}(x)=\dfrac{x}{|x|}\phi'_{\delta\varepsilon}\left(|x|\right)$;
	\item  $0\leq \left|\phi'_{\delta\varepsilon}(x)\right|\leq 1$;
	\item  $|x|\leq \varepsilon+\phi_{\delta\varepsilon}(x)$;
	\item  $\dfrac{\phi'_{\delta\varepsilon}(|x|)}{|x|}\leq \dfrac{\delta}{\varepsilon}$;
	\item  $\phi''_{\delta\varepsilon}\left(|x|\right)=\psi_{\delta\varepsilon}(|x|)\leq\dfrac{2}{|x| \log\delta}\mathbf{1}_{\left[\frac{\varepsilon}{\delta},\varepsilon\right]}(|x|)\leq \dfrac{2\delta}{\varepsilon\log \delta}$.
\end{enumerate}


\begin{proof}[Proof of Proposition \ref{Prop3.1} ]

For each $H>0$, we define a tamed-adaptive Milstein discretisation of equation \eqref{eqn1}  as follows
\begin{equation*}
\begin{cases} 
t_0^H=0, \quad \widehat{X}_0^H=x_0, \quad t_{k+1}^H=t_k^H+h^H(\widehat{X}^H_{t_k^H})\Delta,\\ 
\widehat{X}^H_{t_{k+1}^H}= \widehat{X}^H_{t_k^H}+\mu^H(\widehat{X}^H_{t_k^H})\left(t_{k+1}^H-t_k^H\right)+\sm^H(\widehat{X}^H_{t_k^H})(W_{t_{k+1}^H}-W_{t_{k}^H}) + \dfrac{1}{2}q_{\Delta}(\widehat{X}^H_{t_k^H})\left( (W_{t_{k+1}^H}-W_{t_k^H})^2 -(t_{k+1}^H-t_k^H) \right),
\end{cases} 
\end{equation*}
where 
\begin{align*}
 h^H(x)=\begin{cases}
h(x) & \text{if } |x| \leq H,\\
\dfrac{1}{1+H}& \text{if } |x|>H,
\end{cases}
\quad
 \mu^H(x)=\begin{cases}
\mu(x) & \text{if } |x| \leq H,\\
\dfrac{x}{1+x^2}+\mu(0)& \text{if } |x|>H,
\end{cases}
\quad 
 \sm^H(x)=\begin{cases}
\sm(x) & \text{if } |x| \leq H,\\
1& \text{if } |x|>H.
\end{cases}
\end{align*}
Then, it can be checked that there exists a positive constant $L$, which does not depend on $H$, such that 
\begin{itemize}
\item[\textbf{H1.}]$|\mu^H(x)| \vee |\sm^H(x)| \leq L(1+|x|^{l + \alpha+1})$ for any $x \in \mathbb{R}$;
\item[\textbf{H2.}]$x(\mu^H(x)-\mu(0)) \le L x^2$ for any $x \in \mathbb{R}$;
\item[\textbf{H3.}]$|\mu^H(x)| \le L$ if $|x|>H.$ 
\end{itemize}
Since $\inf_{x \in \mathbb R}|h(x)| > 0$,   $t^H_k\uparrow \infty$ as $k \uparrow \infty.$  
We now define $\underline{t}^H:= \max\{t^H_k: t^H_k \leq t\}.$ Note that $\underline{t}^H $ is a stopping time. 
Let \textcolor{black}{$\widehat{X}^H=(\widehat{X}^H_{t})_{t \geq 0}$} define as 
\begin{align*}
\widehat{X}^H_{t}= x_0+ \int_0^t \mu^H(\widehat{X}^H_{\underline{s}^H})ds+\int_0^t \left(\sm^H(\widehat{X}^H_{\underline{s}^H})+q_{\Delta}(\widehat{X}^H_{\underline{s}^H}) \left(W_s-W_{\underline{s}^H}\right)\right)dW_s. 
\end{align*}
Then, applying It\^o's formula to the function $\phi_{\delta \epsilon}(\widehat{X}^H_{t}),$ we have that for any $t \geq 0$,
\begin{align*}
\phi_{\delta \epsilon}(\widehat{X}^H_{t})
&=\phi_{\delta \epsilon}(x_0)+ \int_0^t \phi'_{\delta \epsilon}(\widehat{X}^H_{s})\mu^H(\widehat{X}^H_{\underline{s}^H})ds+\dfrac{1}{2}\int_0^t \phi''_{\delta \epsilon}(\widehat{X}^H_{s})\left(\sm^H(\widehat{X}^H_{\underline{s}^H})+q_{\Delta}(\widehat{X}^H_{\underline{s}^H}) \left(W_s-W_{\underline{s}^H}\right)\right)^2 ds + M_t,
\end{align*}
where $M_t$ is the stochastic integral part defined by 
$M_t = \int_0^t \phi'_{\delta \epsilon}(\widehat{X}^H_{s})\left(\sm^H(\widehat{X}^H_{\underline{s}^H})+q_{\Delta}(\widehat{X}^H_{\underline{s}^H}) \left(W_s-W_{\underline{s}^H}\right)\right) dW_s$. Then, using the inequality $(a+b)^2 \leq 2(a^2 + b^2)$, we have 
\begin{align} \label{ie1}
\phi_{\delta \epsilon}(\widehat{X}^H_{t})
&\leq  \phi_{\delta \epsilon}(x_0)+ \int_0^t \left( \phi'_{\delta \epsilon}(\widehat{X}^H_{s})
- \phi'_{\delta \epsilon}(\widehat{X}^H_{\underline{s}^H})\right) \mu^H(\widehat{X}^H_{\underline{s}^H})ds \notag\\
&\qquad+
 \int_0^t \phi'_{\delta \epsilon}(\widehat{X}^H_{\underline{s}^H})\left(\mu^H(\widehat{X}^H_{\underline{s}^H}) - \mu^H(0)\right)ds + \int_0^t \phi'_{\delta \epsilon}(\widehat{X}^H_{\underline{s}^H})\mu^H(0)ds \notag \\
&\qquad+\int_0^t \phi''_{\delta \epsilon}(\widehat{X}^H_{s})(\sm^H(\widehat{X}^H_{\underline{s}^H}))^2ds+ \int _0^t \phi''_{\delta \epsilon}(\widehat{X}^H_{s})q_{\Delta}^2(\widehat{X}^H_{\underline{s}^H}) \left(W_s-W_{\underline{s}^H}\right)^2 ds + M_t.
\end{align}
It follows from \textbf{YW5} that 
\begin{align}\label{ie2}
\phi'_{\delta \epsilon}(\widehat{X}^H_{s})
- \phi'_{\delta \epsilon}(\widehat{X}^H_{\underline{s}^H}) \leq \dfrac{2\delta}{\epsilon \log \delta}|\widehat{X}^H_s- \widehat{X}^H_{\underline{s}^H}|.
\end{align}
From \textbf{YW1, YW2, H2}, we get that
\begin{align}\label{ie3}
\phi'_{\delta \epsilon}(\widehat{X}^H_{\underline{s}^H})\left(\mu^H(\widehat{X}^H_{\underline{s}^H}) - \mu^H(0)\right)&=\dfrac{\phi'_{\delta \epsilon}(|\widehat{X}^H_{\underline{s}^H}|)}{|\widehat{X}^H_{\underline{s}^H}|}\widehat{X}^H_{\underline{s}^H}\left(\mu^H(\widehat{X}^H_{\underline{s}^H}) - \mu^H(0)\right)  \leq L|\widehat{X}^H_{\underline{s}^H}| \leq L|\widehat{X}^H_{s}|+L|\widehat{X}^H_{s}-\widehat{X}^H_{\underline{s}^H}|.
\end{align}
It follows from \textbf{H1} that 
\begin{align*}
|\sm^H(\widehat{X}^H_{\underline{s}^H})|^2 &\leq L^2(1+|\widehat{X}^H_{\underline{s}^H}|^{l + \alpha+1})^2 \leq L^2 2^{2(l+\alpha+1)} ( 1+ |\widehat{X}^H_s|^{2(l+\alpha+1)} + |\widehat{X}^H_s- \widehat{X}^H_{\underline{s}^H}|^{2(l+\alpha+1)}).
\end{align*}
Using \textbf{YW5}, there exists a constant $C_1$ such that 
\begin{align}\label{ie4}
\phi''_{\delta \epsilon}(\widehat{X}^H_{s})(\sm^H(\widehat{X}^H_{\underline{s}^H}))^2
\leq C_1(1+|\widehat{X}^H_s-\widehat{X}^H_{\underline{s}^H}|^{2(l + \alpha+1)}).
\end{align}
Inserting   \eqref{ie2}, \eqref{ie3}, \eqref{ie4}  into \eqref{ie1} and using \eqref{C2}, \textbf{YW2, YW3, YW5}, we get that for any stopping time $\tau \leq T$,
\begin{align*}
\left|\widehat{X}^H_{t\wedge \tau} \right| &\leq \varepsilon+\phi_{\delta\varepsilon}(\widehat{X}^H_{t\wedge \tau})\\
&\leq \epsilon + \phi_{\delta \epsilon}(x_0)+ \dfrac{2\delta}{\epsilon \log \delta}\int_0^{t\wedge \tau}|\widehat{X}^H_s- \widehat{X}^H_{\underline{s}^H}||\mu^H(\widehat{X}^H_{\underline{s}^H})|ds +\int_0^{t\wedge \tau}\left(L|\widehat{X}^H_{s}|+L|\widehat{X}^H_{s}-\widehat{X}^H_{\underline{s}^H}|\right)ds
\notag\\
&\quad+C_1\int_0^{t\wedge \tau} \left(1+|\widehat{X}^H_s- \widehat{X}^H_{\underline{s}^H}|^{2(l + \alpha+1)}\right)ds+\dfrac{2\delta}{\epsilon \log \delta}\Big(\dfrac{L}{\sqrt\Delta}\Big)^2\int_0^{t\wedge \tau}\left(W_s-W_{\underline{s}^H}\right)^2 ds 
+|\mu(0)|(t\wedge \tau) + M_{t \wedge \tau}.
\end{align*}
This yields to
\begin{align}\label{ex^}
\bE\left[\left|\widehat{X}^H_{t\wedge \tau} \right|\right] \leq &\epsilon + |x_0|+L\int_0^t \bE\left[\left|\widehat{X}^H_{s\wedge \tau} \right|\right]ds + \dfrac{2\delta}{\epsilon \log \delta}\int_0^t \bE\left[|\widehat{X}^H_s- \widehat{X}^H_{\underline{s}^H}||\mu^H(\widehat{X}^H_{\underline{s}^H})|\right]ds +L\int_0^t \bE\left[|\widehat{X}^H_{s}-\widehat{X}^H_{\underline{s}^H}|\right]ds\notag\\
& + |\mu(0)|t +C_1 \int_0^t \left(1+\bE\left[|\widehat{X}^H_s- \widehat{X}^H_{\underline{s}^H}|^{2(l + \alpha+1)}\right]\right)ds
+\dfrac{2\delta}{\epsilon \log \delta}\Big(\dfrac{L}{\sqrt\Delta}\Big)^2\int_0^t \bE\left[\left(W_s-W_{\underline{s}^H}\right)^2\right] ds.
\end{align}
Now, we write
\begin{align}\label{diffx^}
\left|\widehat{X}^H_s- \widehat{X}^H_{\underline{s}^H}\right|=\left| \mu^H(\widehat{X}^H_{\underline{s}^H})(s- \underline{s}^H)+\sm^H(\widehat{X}^H_{\underline{s}^H})\left(W_s-W_{\underline{s}^H}\right)+\dfrac{1}{2}q_{\Delta}(\widehat{X}^H_{\underline{s}^H}) \left((W_s-W_{\underline{s}^H})^2-(s- \underline{s}^H)\right)\right|.
\end{align}
We will use the following moment property of the Brownian motion due to its strong Markov property.  For any $q>0$ and $s\geq 0$, 
\begin{align} \label{markov2} 
\mathbb{E}\left[\left(W_s-W_{\underline{s}^H}\right)^q\big| \mathcal{F}_{\underline{s}^H}\right]= \begin{cases} 0 & \text{ if } q \text{ is an odd integer}\textcolor{black}{,}\\  C_q(s-\underline{s}^H)^{q/2}& \text{ if } q \text{ is an even integer}, \end{cases}
\end{align} 
for some positive constant $C_q$. 
Thus, 
\begin{align}\label{bx^}
&\bE\left[|\widehat{X}^H_s- \widehat{X}^H_{\underline{s}^H}||\mu^H(\widehat{X}^H_{\underline{s}^H})|\right]=\bE\left[\bE\left[|\widehat{X}^H_s- \widehat{X}^H_{\underline{s}^H}||\mu^H(\widehat{X}^H_{\underline{s}^H})||\mf_{\underline{s}^H}\right]\right]\notag \\
&\leq \bE\left[(\mu^H(\widehat{X}^H_{\underline{s}^H}))^2(s- \underline{s}^H)\right]+\bE\left[|\mu^H(\widehat{X}^H_{\underline{s}^H})\sm^H(\widehat{X}^H_{\underline{s}^H})|\bE\left[|W_s-W_{\underline{s}^H}||\mf_{\underline{s}^H}\right]\right] \notag\\
&\qquad+\bE\left[\Big|\dfrac{1}{2}\mu^H(\widehat{X}^H_{\underline{s}^H})q_{\Delta}(\widehat{X}^H_{\underline{s}^H})\Big|\bE\left[ |(W_s-W_{\underline{s}^H})^2-(s- \underline{s}^H)||\mf_{\underline{s}^H}\right]\right] \notag\\
&\leq \bE[(\mu^H(\widehat{X}^H_{\underline{s}^H}))^2(s- \underline{s}^H)]+ \bE\left[|\mu^H(\widehat{X}^H_{\underline{s}^H})\sm^H(\widehat{X}^H_{\underline{s}^H})|(s-\underline{s}^H)^{1/2}\right]+\bE\left[\Big|\mu^H(\widehat{X}^H_{\underline{s}^H})q_{\Delta}(\widehat{X}^H_{\underline{s}^H})\Big|(s- \underline{s}^H)\right].
\end{align}
From \eqref{chooseh} and the definition of $\mu^H(x), \sm^H(x)$, there exists a constant $C_2$ such that for all $s \in [0,T],$
\begin{align} \label{dk}
 \max \left\{(\mu^H(\widehat{X}^H_{\underline{s}^H}))^2(s- \underline{s}^H),|\mu^H(\widehat{X}^H_{\underline{s}^H})\sm^H(\widehat{X}^H_{\underline{s}^H})|^2(s- \underline{s}^H), |\mu^H(\widehat{X}^H_{\underline{s}^H})(s- \underline{s}^H)|\right\} \leq C_2 \Delta .
 \end{align} 
 This, together with \eqref{C2} and  \eqref{bx^} implies that there exists a constant $C_3$ such that 
 \begin{align} \label{Eb^}
\sup_{s \in [0,T]} \bE\left[|\widehat{X}^H_s- \widehat{X}^H_{\underline{s}^H}||\mu^H(\widehat{X}^H_{\underline{s}^H})|\right] \leq C_3 \Delta.
 \end{align}
Furthermore, from \eqref{diffx^} and \eqref{markov2}, we get that for any $q \geq 1$, there exists a constant $C_4(q)$ such that, for all $s \in [0,T]$,
\begin{align*}
&\bE\left[|\widehat{X}^H_s- \widehat{X}^H_{\underline{s}^H}|^q\right]=\bE\left[\bE\left[|\widehat{X}^H_s- \widehat{X}^H_{\underline{s}^H}|^q |\mf_{\underline{s}^H}\right]\right]\notag \\
& \leq 3^{q-1}\bigg(\bE\left[\left| \mu^H(\widehat{X}^H_{\underline{s}^H})(s- \underline{s}^H)\right|^q\right]+\bE\left[\left|\sm^H(\widehat{X}^H_{\underline{s}^H})\right|^q \bE\left[|W_s-W_{\underline{s}^H}|^q|\mf_{\underline{s}^H}\right]\right] \notag\\
&\qquad+\bE\left[\left|\dfrac{1}{2}q_{\Delta}(\widehat{X}^H_{\underline{s}^H})\right|^q \bE\left[\left|(W_s-W_{\underline{s}^H})^2-(s- \underline{s}^H)\right|^q|\mf_{\underline{s}^H}\right]\right]\bigg) \notag\\
&\leq C_4(q)\left(\bE\left[\left| \mu^H(\widehat{X}^H_{\underline{s}^H})(s- \underline{s}^H)\right|^q\right]+\bE\left[|\sm^H(\widehat{X}^H_{\underline{s}^H})|^q (s- \underline{s}^H)^{q/2}\right]+\bE\left[\left|q_{\Delta}(\widehat{X}^H_{\underline{s}^H})\right|^q (s- \underline{s}^H)^q\right]\right).
\end{align*}
This, together with \eqref{dk} and \eqref{C2}, implies that,  for any  $q \geq 1,$ there exists a constant $C_5(q)$ such that 
\begin{align}\label{ediffx^}
 \bE[|\widehat{X}^H_s- \widehat{X}^H_{\underline{s}^H}|^q] \leq C_5(q). 
\end{align}
Choosing $q=1$ and $q=2(l + \alpha+1)$ in \eqref{ediffx^}, plugging \eqref{Eb^} and \eqref{ediffx^} into \eqref{ex^}, we get, for any $t\in [0,T]$
 \begin{align*}
\bE\left[\left|\widehat{X}^H_{t\wedge \tau} \right|\right] \leq C_6 +L\int_0^{t }\bE\left[\left|\widehat{X}^H_{s\wedge \tau} \right|\right]ds,
\end{align*}
for some constant $C_6$.
Then, using Gronwall's inequality, we obtain that for any $t \leq T$ and  stopping time $\tau \leq T$,
$\bE\left[\left|\widehat{X}^H_{t\wedge \tau} \right|\right] \leq C_6 e^{Lt}.$
Choosing $t=T,$ we get that for any stopping time $\tau \leq T$,
$\bE\left[\left|\widehat{X}^H_{\tau} \right|\right] \leq C_6 e^{LT}.$
Thus, applying Proposition IV.4.7 in \cite{RY}, there exists a constant $C_7$ such that 
\begin{align}\label{Esup}
\sup_{H>0}\bE\left[\sup_{0 \leq t\leq T}\sqrt{\left|\widehat{X}^H_{t} \right|}\right] \leq C_7. 
\end{align}
Now, for any $T>0,$ we write
\begin{align*}
\p\left(t_k \leq T\right)=\p\left(t_k \leq T, \displaystyle\sup_{0 \leq t\leq T}\sqrt{\left|\widehat{X}^H_{t} \right|} \leq \sqrt{H}\right)+\p\left(t_k \leq T, \sup_{0 \leq t\leq T}\sqrt{\left|\widehat{X}^H_{t} \right|} > \sqrt{H}\right).
\end{align*}
Note that, on the set $\lbrace \sup_{0 \leq t \leq T}\sqrt{|\widehat{X}^H_{t} |} \leq \sqrt{H}\rbrace$ we have $ t^H_k=t_k$ if $t_k \leq T.$ Therefore, 
\begin{align*}
\limsup_{k \rightarrow \infty} \p\left(t_k \leq T, \displaystyle\sup_{0 \leq t\leq T}\sqrt{\left|\widehat{X}^H_{t} \right|} \leq \sqrt{H}\right) \leq \limsup_{k \rightarrow \infty} \p\left(t_k^H \leq T\right)=0.
\end{align*}
Moreover, using Markov's inequality and \eqref{Esup}, we obtain
\begin{align*}
\p\left(t_k \leq T, \sup_{0 \leq t\leq T}\sqrt{\left|\widehat{X}^H_{t} \right|} > \sqrt{H}\right) &\leq \p\left(\sup_{0 \leq t\leq T}\sqrt{\left|\widehat{X}^H_{t} \right|} > \sqrt{H}\right) \leq \dfrac{\bE\left[\sup_{0 \leq t\leq T}\sqrt{\left|\widehat{X}^H_{t} \right|}\right]}{\sqrt{H}} \leq \dfrac{C_7}{\sqrt{H}}.
\end{align*}
Consequently, we have shown that
$\limsup_{k\rightarrow \infty}\p(t_k \leq T) \leq \dfrac{C_7}{\sqrt{H}},$
for any $H>0.$ Letting $H \uparrow \infty,$ we get that $\displaystyle\limsup_{k\rightarrow \infty}\p(t_k \leq T)=0$ for any $T>0.$  This implies that $\displaystyle\lim_{k \rightarrow \infty}t_k=+\infty$ almost surely. Thus, the desired result follows.
\end{proof}

\subsection{Proof of Proposition \ref{mmX^}}\label{proof moment}

We employ the induction method to show \eqref{EXtmu2}. First, we show that \eqref{EXtmu2} holds for $k=1$. Indeed, applying  It\^o's formula to $e^{-2\gamma t}\widehat{X}_t^2$, we have 
	\begin{align*}
	e^{-2\gamma t}\widehat{X}_t^2 &
	=x_0^2+ 2\int_{0}^{t} e^{-2\gamma s}\left(-\gamma \widehat{X}_s^2 +  \widehat{X}_s \mu(\widehat{X}_{\underline{s}}) +\dfrac{1}{2}\left( \sigma(\widehat{X}_{\underline{s}})+ q_\Delta(\widehat{X}_{\underline{s}})(W_s- W_{\underline{s}})\right)^2\right) ds\notag \\
	& \hspace*{0.8cm}+ 2\int_{0}^{t} e^{-2\gamma s}\widehat{X}_s \left(\sm(\widehat{X}_{\underline{s}}) + q_\Delta(\widehat{X}_{\underline{s}})(W_s- W_{\underline{s}})\right)dW_s.  
	\end{align*}
Since \textbf{A1} holds for some  $p_0 \geq 2$, we have $-\gamma \widehat{X}_{\underline{s}}^2 +  \widehat{X}_ {\underline{s}}\mu(\widehat{X}_{\underline{s}}) +\dfrac{1}{2} \sigma^2(\widehat{X}_{\underline{s}}) \leq \eta$. Thus, 
 \begin{align*}
-\gamma \widehat{X}_s^2 &+  \widehat{X}_s \mu(\widehat{X}_{\underline{s}}) +\dfrac{1}{2}\left( \sigma(\widehat{X}_{\underline{s}})+ q_\Delta(\widehat{X}_{\underline{s}})(W_s- W_{\underline{s}})\right)^2 \\
\leq & \eta + |\gamma ||\widehat{X}_s^2- \widehat{X}_{\underline{s}}^2 |+ |\widehat{X}_s- \widehat{X}_ {\underline{s}}||\mu(\widehat{X}_{\underline{s}})| +\dfrac{1}{2} q^2_{\Delta}(\widehat{X}_{\underline{s}})(W_s- W_{\underline{s}})^2+ |\sigma(\widehat{X}_{\underline{s}})| |q_\Delta(\widehat{X}_{\underline{s}})||W_s- W_{\underline{s}}|.
 \end{align*}
For each $H>0,$ set $\tau_H:=\inf\{t \geq 0:  |\widehat{X}_t| \geq H\}.$ We have
	\begin{align}\label{e^Xtau}
	&e^{-2\gamma ( t \wedge \tau_H)}\widehat{X}_{t\wedge \tau_H}^2 
	 \leq x_0^2+ 2\eta\int_{0}^{t \wedge\tau_H} e^{-2\gamma s} ds + 2\int_{0}^{t \wedge \tau_H} e^{-2\gamma s}\widehat{X}_s \left(\sm(\widehat{X}_{\underline{s}}) + q_\Delta(\widehat{X}_{\underline{s}})(W_s- W_{\underline{s}})\right)dW_s\notag \\
	 &\quad+ 2\int_{0}^{t} e^{-2\gamma s}\left(|\gamma ||\widehat{X}_s^2- \widehat{X}_{\underline{s}}^2 |+ |\widehat{X}_s- \widehat{X}_ {\underline{s}}||\mu(\widehat{X}_{\underline{s}})| +\dfrac{1}{2} q^2_{\Delta}(\widehat{X}_{\underline{s}})(W_s- W_{\underline{s}})^2+ |\sigma(\widehat{X}_{\underline{s}})| |q_\Delta(\widehat{X}_{\underline{s}})||W_s- W_{\underline{s}}|\right) ds.
	\end{align}

Thanks to  Lemma \ref{Lem:-gmX^p1}, Lemma \ref{Lem:X^p-1b1}, Lemma \ref{Lem:$X^pq^21$} with $p=2$, there exists a constant $C_0$,  which does not depend on $\Delta$, such that 
   \begin{align}\label{ep}
	\sup_{s>0} \bE\left[|\gamma ||\widehat{X}_s^2- \widehat{X}_{\underline{s}}^2 |+ |\widehat{X}_s- \widehat{X}_ {\underline{s}}||\mu(\widehat{X}_{\underline{s}})| +\dfrac{1}{2} q^2_{\Delta}(\widehat{X}_{\underline{s}})(W_s- W_{\underline{s}})^2+ |\sigma(\widehat{X}_{\underline{s}})| |q_\Delta(\widehat{X}_{\underline{s}})||W_s- W_{\underline{s}}|\right] \leq C_0.
	\end{align}

Taking expectation in \eqref{e^Xtau} and using \eqref{ep}, we get that
	\begin{align*}
	\mathbb{E}\left[e^{-2\gamma (t\wedge \tau_H)}\widehat{X}_{t \wedge \tau_H}^2\right]&
	\leq x_0^2+2(\eta+C_0) \int_0^{t} e^{-2\gamma s}ds.
	\end{align*}
Thus,
$\p(\tau_H < t) \leq \left(x_0^2+2(\eta+C_0) \int_0^{t} e^{-2\gamma s}ds\right)H^{-2},$
	which deduces that $\tau_H \uparrow \infty $ almost surely as $H \uparrow \infty.$ Then, letting $H \uparrow \infty$ and applying Fatou's Lemma,  we obtain
	\begin{align*}
	\mathbb{E}\left[e^{-2\gamma t}\widehat{X}_{t}^2\right]&
	\leq x_0^2+2(\eta+C_0) \int_0^{t} e^{-2\gamma s}ds.
	\end{align*}
	This yields
	\begin{align} \label{p=2,1}
	\bE \left[\widehat{X}_{t}^2\right] \le \begin{cases}  \left(x_{0}^{2}+\dfrac{\eta+C_0}{\gamma}\right) e^{2\gamma t}-\dfrac{\eta+C_0}{\gamma} & \text{ if \;\;} \gamma \not =0,\\ 
	x_{0}^{2}+2\left(\eta+C_0\right) t & \text{ if \;\;} \gamma=0.
	\end{cases} 
	\end{align}		
Consequently, from \eqref{p=2,1} and \eqref{qh1} of Lemma 5.5, we conclude  that \eqref{EXtmu2} holds for $k=1$.

	Next, suppose that  the estimate \eqref{EXtmu2} holds for any $k \leq k_0 \leq \textcolor{black}{\lfloor p_0/2 \rfloor}-1$, we wish to show that \eqref{EXtmu2} still holds for $k=k_0+1$. For this,  applying It\^o's formula for $e^{-p\gamma t}\widehat{X}^p_t$ with even integer $p=2(k_0+1)$, we have
	\begin{align}
	e^{-p\gamma t}\left|\widehat{X}_t\right|^p 
	&=x_0^p+ \int_{0}^{t} e^{-p\gamma s}\left[-p\gamma \widehat{X}_s ^p +  p\widehat{X}_s^{p-1} \mu(\widehat{X}_{\underline{s}}) + \dfrac{p(p-1)}{2} \widehat{X}_s^{p-2}\left(\sm (\widehat{X}_{\underline{s}}) +q_\Delta(\widehat{X}_{\underline{s}})(W_s- W_{\underline{s}})\right)^2 \right]d s + M_t,
 \end{align}
 where $M_t=  p\int_{0}^{t} e^{-p\gamma s} \widehat{X}_s^{p-1} \left(\sm(\widehat{X}_{\underline{s}})+q_\Delta(\widehat{X}_{\underline{s}})(W_s- W_{\underline{s}})\right) dW_s.$ 
Similar to the case $k=1$, using assumption \textbf{A1}, we get 
	\begin{align}\label{e^pX^p1}
	e^{-p\gamma (t \wedge \tau_H)}\left|\widehat{X}_{t\wedge \tau_H}\right|^p 
	& \leq x_0^p+  p\eta \int_{0}^{t\wedge \tau_H}e^{-p\gamma s}\widehat{X}_{\underline{s}}^{p-2}d s \notag \\
	&\qquad+p\int_{0}^{t} e^{-p\gamma s}\bigg[|\gamma| |\widehat{X}_s^p-\widehat{X}_{\underline{s}} ^p| +  |\widehat{X}_s^{p-1}-\widehat{X}_{\underline{s}}^{p-1}| |\mu(\widehat{X}_{\underline{s}})|+ \dfrac{p-1}{2} |\widehat{X}_{s}^{p-2}-\widehat{X}_{\underline{s}}^{p-2}|\sm^2 (\widehat{X}_{\underline{s}}) \notag \\
	&\qquad+ \dfrac{p-1}{2} |\widehat{X}_{s}|^{p-2}\left(2|\sm (\widehat{X}_{\underline{s}})| |q_\Delta(\widehat{X}_{\underline{s}})(W_s- W_{\underline{s}})|+ q^2_\Delta(\widehat{X}_{\underline{s}})(W_s- W_{\underline{s}}) ^2 \right)\bigg]d s +  M_{t \wedge \tau_H}. 
	\end{align}
Combining Lemma \ref{Lem:-gmX^p1}, Lemma \ref{Lem:X^p-1b1}, Lemma \ref{Lem:X^p-2sm1} and Lemma \ref{Lem:$X^pq^21$}, we get that
	\begin{align*}
	&\bE \left[|\gamma| |\widehat{X}_s^p-\widehat{X}_{\underline{s}} ^p| +  |\widehat{X}_s^{p-1}-\widehat{X}_{\underline{s}}^{p-1}| |\mu(\widehat{X}_{\underline{s}})|+ \dfrac{p-1}{2} |\widehat{X}_{s}^{p-2}-\widehat{X}_{\underline{s}}^{p-2}|\sm^2 (\widehat{X}_{\underline{s}}) \right. \\
	&\qquad\left.+ \dfrac{p-1}{2} |\widehat{X}_{s}|^{p-2}\left(2|\sm (\widehat{X}_{\underline{s}})| |q_\Delta(\widehat{X}_{\underline{s}})(W_s- W_{\underline{s}})|+ q^2_\Delta(\widehat{X}_{\underline{s}})(W_s- W_{\underline{s}}) ^2 \right) \Big|\mathcal{F}_{\underline{s}} \right] \notag \\
	&\le C(p, \gamma)\sum\limits_{i=0}^{p-2} |\widehat{X}_{\underline{s}}|^i.
	\end{align*}
This, combined with \eqref{e^pX^p1}, deduces that
	\begin{align*}
	\bE\left[e^{-p\gamma (t \wedge \tau_H)}\left|\widehat{X}_{t\wedge \tau_H}\right|^p\right] 
&\leq x_0^p+ p\eta\bE\left[\int_{0}^{t\wedge \tau_H}  e^{-p\gamma s}\widehat{X}_{\underline{s}}^{p-2}d s \right]+ C(p, \gamma) \bE\left[\int_{0}^{t} e^{-p\gamma s}\sum\limits_{i=0}^{p-2} |\widehat{X}_{\underline{s}}|^i d s\right] \notag \\
	&\qquad + p\bE\left[\int_{0}^{t\wedge \tau_H} e^{-p\gamma s} \widehat{X}_s^{p-1} \left(\sm(\widehat{X}_{\underline{s}})+q_\Delta(\widehat{X}_{\underline{s}})(W_s- W_{\underline{s}})\right) dW_s\right].
	\end{align*}
Using the same argument as in the case $p=2$ by letting $H \uparrow \infty$ and  applying Fatou's Lemma,  we obtain
\begin{align}\label{tag421}
&\bE\left[e^{-p\gamma t}\left|\widehat{X}_{t}\right|^p\right]\leq  
	x_0^p+C(p,\eta, \gamma) \int_{0}^{t} e^{-p\gamma s} \sum_{i=0}^{p-2}\mathbb{E}\left[ \left|\widehat{X}_{\underline{s}}\right|^{i}\right]d s.
\end{align}
	 Then, from the estimates  \eqref{qh1} of Lemma 5.5,  \eqref{tag421}  and the inductive assumption, we conclude that \eqref{EXtmu2} holds for $k=k_0+1$. Thus,  the desired result follows.	

\subsection{Proof of Proposition \ref{ulN}}
It follows from \eqref{chooseh} that there is a positive constant $C$ such that 
$h(x)^{-1} \leq C(1 + |x|^{4(l + \alpha +1)} + |x|^{2l_0}). $
Thus, 
$$
N(T) \leq \int_0^T \frac{1}{h (\widehat{X}_{\underline{t}})\Delta} \mathrm{d} t+1 \leq \frac{2C}{\Delta} \int_0^T  (1+ |\widehat{X}_{\underline{s}}|^{2l_0} +  |\widehat{X}_{\underline{s}}|^{4(l+\alpha+1)})ds + 1.$$
Therefore, 
$$\mathbb{E}\left[\left|N(T)-1\right|\right] \leq \frac{2CT}{\Delta} \sup_{0\leq s \leq T}\bE\left[ 1+ |\widehat{X}_{\underline{s}}|^{2l_0} + |\widehat{X}_{\underline{s}}|^{4(l + \alpha+1)}\right].$$

\subsection{Proof of Theorem \ref{main Thr }} \label{proof3}
We write 
\begin{align*}
X_t- \widehat{X}_t= \int_0^t(\mu(X_s)- \mu(\widehat{X}_{\underline{s}})) ds+ \int_0^t\left(\sm(X_s)- \sm(\widehat{X}_{\underline{s}})-q_{\Delta}(\widehat{X}_{\underline{s}}) (W_s- W_{\underline{s}})\right)dW_s.
\end{align*}
	For any  $a \in \mathbb{R},$ applying It\^o's formula to $e^{-at}|X_t- \widehat{X}_t|^2,$ we have
	\begin{align}\label{ItoXX^}	
	&e^{-at}|X_t- \widehat{X}_t|^2 = \int_0^t e^{-as} g_s ds +\int_0^t 2e^{-as}(X_s- \widehat{X}_s)\left(\sm(X_s)- \sm(\widehat{X}_{\underline{s}})-q_{\Delta}(\widehat{X}_{\underline{s}})(W_s- W_{\underline{s}})\right)dW_s,
	\end{align}
	where  $g_s :=  -a	|X_s- \widehat{X}_s|^2+2(X_s- \widehat{X}_s)(\mu(X_s)- \mu(\widehat{X}_{\underline{s}}))+\left(\sm(X_s)- \sm(\widehat{X}_{\underline{s}})-q_{\Delta}(\widehat{X}_{\underline{s}}) (W_s- W_{\underline{s}})\right)^2 $.
	First, observe that
	\begin{align} \label{diffxmu}
	&(X_s- \widehat{X}_s)\left(\mu(X_s)- \mu(\widehat{X}_{\underline{s}})\right) \notag \\
	&= (X_s- \widehat{X}_s)\left(\mu(X_s)- \mu(\widehat{X}_s)\right)+(X_s- \widehat{X}_s)\left(\mu(\widehat{X}_s)- \mu(\widehat{X}_{\underline{s}})- \mu'(\widehat{X}_{\underline{s}})(\widehat{X}_s- \widehat{X}_{\underline{s}})\right)\notag\\
	&\quad+(X_s- \widehat{X}_s)\mu'(\widehat{X}_{\underline{s}})\left(\widehat{X}_s- \widehat{X}_{\underline{s}}-\sm(\widehat{X}_{\underline{s}})(W_s- W_{\underline{s}})\right)+(X_s- \widehat{X}_s-X_{\underline{s}}+\widehat{X}_{\underline{s}})\mu'(\widehat{X}_{\underline{s}})\sm(\widehat{X}_{\underline{s}})(W_s- W_{\underline{s}}) \notag \\
	&\quad+(X_{\underline{s}}-\widehat{X}_{\underline{s}})\mu'(\widehat{X}_{\underline{s}})\sm(\widehat{X}_{\underline{s}})(W_s- W_{\underline{s}}).
	\end{align}
	For any $\epsilon>0,$ we have
	\begin{align}
	&(X_s- \widehat{X}_s)\left(\mu(\widehat{X}_s)- \mu(\widehat{X}_{\underline{s}})- \mu'(\widehat{X}_{\underline{s}})(\widehat{X}_s- \widehat{X}_{\underline{s}})\right) +(X_s- \widehat{X}_s)\mu'(\widehat{X}_{\underline{s}})\left(\widehat{X}_s- \widehat{X}_{\underline{s}}-\sm(\widehat{X}_{\underline{s}})(W_s- W_{\underline{s}})\right) \notag\\ 
	&\leq \epsilon |X_s- \widehat{X}_s|^2
	+\dfrac{1}{2\epsilon}\left|\mu(\widehat{X}_s)- \mu(\widehat{X}_{\underline{s}})- \mu'(\widehat{X}_{\underline{s}})(\widehat{X}_s- \widehat{X}_{\underline{s}})\right|^2 
	+\dfrac{1}{2\epsilon}|\widehat{X}_s- \widehat{X}_{\underline{s}}-\sm(\widehat{X}_{\underline{s}})(W_s- W_{\underline{s}})|^2|\mu'(\widehat{X}_{\underline{s}})|^2. \label{eqn:mprof1}
	\end{align}
	Using \eqref{C^1+alpha} and the estimate $|\widehat{X}_s|^l  \leq 2^{l-1}(|\widehat{X}_s - \widehat{X}_{\underline{s}}|^l  + |\widehat{X}_{\underline{s}}|^l)$, we get  
	\begin{align*}
	\left|\mu(\widehat{X}_s)- \mu(\widehat{X}_{\underline{s}})- \mu'(\widehat{X}_{\underline{s}})(\widehat{X}_s- \widehat{X}_{\underline{s}})\right|^2& \leq C_1 (1+|\widehat{X}_s|^{l}+|\widehat{X}_{\underline{s}}|^{l})^2|\widehat{X}_s - \widehat{X}_{\underline{s}}|^{2(1+\alpha)}\\
	& \leq C_2 (1 + |\widehat{X}_{\underline{s}}|^{2l})|\widehat{X}_s - \widehat{X}_{\underline{s}}|^{2(1+\alpha)} + C_2|\widehat{X}_s - \widehat{X}_{\underline{s}}|^{2(1+\alpha + l)}.
	\end{align*}
	Applying Lemma \ref{lm5},  we have
	\begin{align}\label{E1}
	\bE\left[\left|\mu(\widehat{X}_s)- \mu(\widehat{X}_{\underline{s}})- \mu'(\widehat{X}_{\underline{s}})(\widehat{X}_s- \widehat{X}_{\underline{s}})\right|^2 |\mf_{\underline{s}}\right] \leq C_3 \left(1+|\widehat{X}_{\underline{s}}|^{2l}\right)  \Delta^{1+\alpha} + 	C_3  \Delta^{1+\alpha+ l}. 	
	\end{align}
	Next, using  Lemma \ref{lm8},  estimate  \eqref{U2} and the fact that $|\mu'(\widehat{X}_{\underline{s}})\sm(\widehat{X}_{\underline{s}})|(s-\underline{s}) \leq h_0 \Delta,$ we get that, for any $\epsilon>0,$ 
	\begin{align} \label{E2}
	&\left| \bE\left[(X_s- \widehat{X}_s-X_{\underline{s}}+\widehat{X}_{\underline{s}})\mu'(\widehat{X}_{\underline{s}})\sm(\widehat{X}_{\underline{s}})(W_s- W_{\underline{s}})|\mf_{\underline{s}}\right] \right | \notag\\
	&=   |\mu'(\widehat{X}_{\underline{s}})\sm(\widehat{X}_{\underline{s}})| \left|\bE\left[(X_s- \widehat{X}_s-X_{\underline{s}}+\widehat{X}_{\underline{s}})(W_s- W_{\underline{s}})|\mf_{\underline{s}}\right]\right| \notag\\
	& \leq |\mu'(\widehat{X}_{\underline{s}})\sm(\widehat{X}_{\underline{s}})|\left(C(s- \underline{s})^{\frac{3+\alpha}{2}}(1+|X_{\underline{s}}|^{l + (1+\alpha)(l + \alpha + 1)}) + |\sm(X_{\underline{s}})- \sm(\widehat{X}_{\underline{s}})|(s-\underline{s})\right) \notag\\
	& \leq C_4 \Delta^{\frac{3+\alpha}{2}}\left(1+|X_{\underline{s}}|^{l + (1+\alpha)(l + \alpha + 1)}\right)+ \dfrac{L}{4\epsilon}\Delta^2\left(1+|X_{\underline{s}}|^{l+ \alpha}+|\widehat{X}_{\underline{s}}|^ {l+ \alpha}\right)^2  + \epsilon|X_{\underline{s}}- \widehat{X}_{\underline{s}}|^2.
	\end{align}
Applying Lemma \ref{lm5}, we get
\begin{align} \label{E3}
\bE\left[|\widehat{X}_s- \widehat{X}_{\underline{s}}-\sm(\widehat{X}_{\underline{s}})(W_s- W_{\underline{s}})|^2|\mf_{\underline{s}}\right] \leq C_5 \Delta^2.
\end{align}
Substituting \eqref{eqn:mprof1}, \eqref{E1}, \eqref{E2}, \eqref{E3} into \eqref{diffxmu} and then applying estimate  \eqref{U6}, we obtain that for any $\epsilon>0$, 
\begin{align} \label{EbX}
&\bE\left[(X_s- \widehat{X}_s)(\mu(X_s)- \mu(\widehat{X}_{\underline{s}}))|\mf_{\underline{s}}\right] \notag\\
& \leq \bE\left[ (X_s- \widehat{X}_s)(\mu(X_s)- \mu(\widehat{X}_s))+ \epsilon|X_s- \widehat{X}_s|^2|\mf_{\underline{s}}\right]
+\dfrac{C_3}{2\epsilon} \Delta^{1+\alpha}(1+|\widehat{X}_{\underline{s}}|^{2l}) + \frac{C_3}{2\epsilon} \Delta^{1 + \alpha + l} \notag\\
&\quad +C_4 \Delta^{\frac{3+\alpha}{2}}(1+|X_{\underline{s}}|^{l + (1+\alpha)(l +\alpha + 1)})
+ \dfrac{L}{4\epsilon}(1+|X_{\underline{s}}|^{l + \alpha}+|\widehat{X}_{\underline{s}}|^{l + \alpha})^2 \Delta^2 + \epsilon|X_{\underline{s}}- \widehat{X}_{\underline{s}}|^2 
+\dfrac{C_5 L^2}{2\epsilon}\left(1+|\widehat{X}_{\underline{s}}|^{l+\alpha}\right)^2\Delta^2 \notag\\
&\leq \bE\left[ (X_s- \widehat{X}_s)(\mu(X_s)- \mu(\widehat{X}_s))+ \epsilon|X_s- \widehat{X}_s|^2|\mf_{\underline{s}}\right]+ \epsilon|X_{\underline{s}}- \widehat{X}_{\underline{s}}|^2 \notag\\
&\quad+C_6  \Delta^{1+\alpha}\left(1+|\widehat{X}_{\underline{s}}|^{2l+2\alpha}+ |X_{\underline{s}}|^{l + (1+\alpha)(1 +\alpha + l)}\right).
\end{align}
Second, observe that
\begin{align*}
&\left(\sm(X_s)-\sm(\widehat{X}_{\underline{s}})- q_{\Delta}(\widehat{X}_{\underline{s}})(W_s- W_{\underline{s}})\right)^2\\
&=\left\vert\sm(X_s)-\sm(\widehat{X}_s)+\sm(\widehat{X}_s)-\sm(\widehat{X}_{\underline{s}})- q_{\Delta}(\widehat{X}_{\underline{s}})(W_s- W_{\underline{s}})\right\vert^2  \\
&=|\sm(X_s)-\sm(\widehat{X}_s)|^2+2\left(\sm(X_s)-\sm(\widehat{X}_s)\right)\left(\sm(\widehat{X}_s)-\sm(\widehat{X}_{\underline{s}})- q_{\Delta}(\widehat{X}_{\underline{s}})(W_s- W_{\underline{s}})\right)\\
&\quad+ |\sm(\widehat{X}_s)-\sm(\widehat{X}_{\underline{s}})- q_{\Delta}(\widehat{X}_{\underline{s}})(W_s- W_{\underline{s}})|^2.
\end{align*}
Using  \eqref{U2}, for any $\epsilon>0,$ we have
\begin{align*}
&2\left(\sm(X_s)-\sm(\widehat{X}_s)\right)\left(\sm(\widehat{X}_s)-\sm(\widehat{X}_{\underline{s}})- q_{\Delta}(\widehat{X}_{\underline{s}})(W_s- W_{\underline{s}})\right)  \notag\\
&\leq 2L(1+|X_s|^{l + \alpha}+|\widehat{X}_s|^{l + \alpha})|X_s- \widehat{X}_s|\left|\sm(\widehat{X}_s)-\sm(\widehat{X}_{\underline{s}})- q_{\Delta}(\widehat{X}_{\underline{s}})(W_s- W_{\underline{s}})\right| \notag\\
&\leq \epsilon |X_s- \widehat{X}_s|^2+\dfrac{L^2}{\epsilon}(1+|X_s|^{l + \alpha}+|\widehat{X}_s|^{l + \alpha})^2\left(\sm(\widehat{X}_s)-\sm(\widehat{X}_{\underline{s}})- q_{\Delta}(\widehat{X}_{\underline{s}})(W_s- W_{\underline{s}})\right)^2 \notag\\
&\leq \epsilon |X_s- \widehat{X}_s|^2+ \dfrac{L^4}{4\epsilon^2}\Delta^{1+\alpha}\left(1+|X_s|^{l + \alpha}+|\widehat{X}_s|^{l + \alpha}\right)^4+ \Delta^{-(1+\alpha)}\left(\sm(\widehat{X}_s)-\sm(\widehat{X}_{\underline{s}})- q_{\Delta}(\widehat{X}_{\underline{s}})(W_s- W_{\underline{s}})\right)^4.
\end{align*}
Thus, applying  Lemma \ref{lm7}, we get that for any $l_0\geq \frac{4l}{3(1+\alpha)}$,
\begin{align}\label{Esq}
&\bE\left[|\sm(X_s)-\sm(\widehat{X}_{\underline{s}})- q_{\Delta}(\widehat{X}_{\underline{s}})(W_s- W_{\underline{s}})|^2 |\mf_{\underline{s}}\right] \notag\\
&\leq \bE\left[|\sm(X_s)-\sm(\widehat{X}_s)|^2+ \epsilon |X_s- \widehat{X}_s|^2|\mf_{\underline{s}}\right]+ \dfrac{L^4}{4\epsilon^2}\Delta^{1+\alpha}\bE\left[(1+|X_s|^{l + \alpha}+|\widehat{X}_s|^{l + \alpha})^4|\mf_{\underline{s}}\right] \notag\\
&\quad+ \Delta^{-(1+\alpha)}\bE\left[\left(\sm(\widehat{X}_s)-\sm(\widehat{X}_{\underline{s}})- q_{\Delta}(\widehat{X}_{\underline{s}})(W_s- W_{\underline{s}})\right)^4|\mf_{\underline{s}}\right] \notag\\
&\quad +\bE\left[ |\sm(\widehat{X}_s)-\sm(\widehat{X}_{\underline{s}})- q_{\Delta}(\widehat{X}_{\underline{s}})(W_s- W_{\underline{s}})|^2|\mf_{\underline{s}}\right] \notag\\
&\leq \bE\left[|\sm(X_s)-\sm(\widehat{X}_s)|^2+ \epsilon |X_s- \widehat{X}_s|^2|\mf_{\underline{s}}\right]+  \dfrac{27L^4}{4\epsilon^2}\Delta^{1+\alpha}\bE\left[1+|X_s|^{4(l+\alpha)}+|\widehat{X}_s|^{4(l+\alpha)}|\mf_{\underline{s}}\right]+C_7 \Delta^{1+\alpha}.
\end{align}
Consequently, using \eqref{ItoXX^}, \eqref{EbX}, \eqref{Esq},  for any $p_0 \geq 4(l + \alpha+1)$, $\epsilon 
>0$ and  $l_0\geq \frac{4l}{3(1+\alpha)}$, we deduce that
\begin{align}
\bE\left[ g_s |\mf_{\underline{s}} \right]  &\leq \bE\left[ -a	|X_s- \widehat{X}_s|^2+2(X_s- \widehat{X}_s)(\mu(X_s)- \mu(\widehat{X}_s))+|\sm(X_s)-\sm(\widehat{X}_s)|^2+3 \epsilon |X_s- \widehat{X}_s|^2|\mf_{\underline{s}}\right] \notag\\
&\quad +2\epsilon|X_{\underline{s}}- \widehat{X}_{\underline{s}}|^2 +\dfrac{27L^4}{4\epsilon^2}\Delta^{1+\alpha}\bE\left[1+|X_s|^{4(l + \alpha)}+|\widehat{X}_s|^{4(l + \alpha)}|\mf_{\underline{s}}\right]  \notag\\
&\quad+C_6  \Delta^{1+\alpha}\left(1+|\widehat{X}_{\underline{s}}|^{2l+2\alpha}+ |X_{\underline{s}}|^{l + (1+\alpha)(1 +\alpha + l)}\right) + C_7\Delta^{1+\alpha}. 
\notag
\end{align}
Thanks to \textbf{A2}, we have 
\begin{align}
\bE\left[ g_s |\mf_{\underline{s}} \right] &\leq  (-a+2\lambda+3\epsilon)\bE\left[|X_s- \widehat{X}_s|^2|\mf_{\underline{s}} \right]+ 2\epsilon|X_{\underline{s}}- \widehat{X}_{\underline{s}}|^2 + \dfrac{27L^4}{4\epsilon^2}\Delta^{1+\alpha}\bE\left[1+|X_s|^{4(l + \alpha)}+|\widehat{X}_s|^{4(l + \alpha)}|\mf_{\underline{s}}\right] \notag\\
 &\quad+C_6  \Delta^{1+\alpha}\left(1+|\widehat{X}_{\underline{s}}|^{2l+2\alpha}+ |X_{\underline{s}}|^{l + (1+\alpha)(1 +\alpha + l)}\right) + C_7 \Delta^{1+\alpha}. \notag
 \end{align}
 Using Lemma \ref{lm9}, we have 
 \begin{align}
\bE\left[ g_s |\mf_{\underline{s}} \right]  &\leq (-a+2\lambda+6\epsilon)	|X_{\underline{s}}- \widehat{X}_{\underline{s}}|^2+(-a+2\lambda+4\epsilon)(s-\underline{s})|\sm(X_{\underline{s}})- \sm(\widehat{X}_{\underline{s}})|^2+ \notag\\
 &\quad+C(a, \epsilon, \lambda) \Delta^2\left(1+|X_{\underline{s}}|^{4(l + \alpha+1)} \right)+ \dfrac{27L^4}{4\epsilon^2}\Delta^{1+\alpha}\bE\left[1+|X_s|^{4(l + \alpha)}+|\widehat{X}_s|^{4(l + \alpha)}|\mf_{\underline{s}}\right] \notag\\
 &\quad+C_6  \Delta^{1+\alpha}\left(1+|\widehat{X}_{\underline{s}}|^{2l+2\alpha}+ |X_{\underline{s}}|^{l + (1+\alpha)(1 +\alpha + l)}\right) + C_7 \Delta^{1+\alpha}. \notag\
\end{align}
Hence, choosing $a=2\lambda+ 6\epsilon$ which implies that $ -a+2\lambda+4\epsilon=-2\epsilon <0$, we get that
\begin{align}
&\bE\left[ g_s|\mf_{\underline{s}} \right] \leq C_8  \Delta^2\left(1+|X_{\underline{s}}|^{4(l + \alpha+1)} \right)+ \dfrac{27L^4}{4\epsilon^2}\Delta^{1+\alpha}\bE\left[1+|X_s|^{4(l + \alpha)}+|\widehat{X}_s|^{4(l + \alpha)}|\mf_{\underline{s}}\right] \notag\\
 &\quad+C_6  \Delta^{1+\alpha}\left(1+|\widehat{X}_{\underline{s}}|^{2l+2\alpha}+ |X_{\underline{s}}|^{l + (1+\alpha)(1 +\alpha + l)}\right) + C_7 \Delta^{1+\alpha}. \notag\
\end{align}
Note that $l + (1+\alpha)(1 +\alpha + l) < 4(l + \alpha +1)$. Consequently, taking expectation in \eqref{ItoXX^}	 and applying Proposition \ref{moment nghiem dung}, Proposition \ref{ulX_ut}, and Proposition \ref{mmX^},  we get that there exists a positive constant $C_T =C(x_0, \gamma, \eta, \alpha, l, l_0, \lambda, p_0, T)$ such that
\begin{align*}
\bE\left[e^{-(2\lambda+ 6\epsilon)t}|X_t- \widehat{X}_t|^2\right] \leq C_T\Delta^{1 + \alpha} \int_0^te^{-(2\lambda+ 6\epsilon)s}ds,
\end{align*}
which yields that
\begin{align*}
\sup_{0 \leq t\leq T}\bE\left[|X_t- \widehat{X}_t|^2\right] \leq C_T\Delta^{1 + \alpha}.
\end{align*}
This shows \eqref{r1}.

Finally, when $\lambda<0, \gamma <0,$ it suffices to choose $\epsilon = \frac{-\lambda}{4}$ which implies that $a=2\lambda+6\epsilon = \frac{\lambda}{2}<0.$ This, combined with Remark \ref{rm1} and Remark \ref{rm2}, deduces that there exists a positive constant $C =C(x_0, \gamma, \eta, \alpha, l, l_0, \lambda, p_0)$ which does not depend on $T$  such that
	\begin{equation} 
	\underset{t \ge 0}{\sup }\; \bE \left[  |\widehat{X}_t - X_t| ^2\right] \le C\Delta^{1+\alpha},
	\end{equation} 
	which proves \eqref{r2}. This concludes our proof.

\section{Numerical experiments} 
\label{sec:nume} 
\subsection{Empirical rates of convergence of the tamed-adaptive scheme} 
In this section, we perform an empirical study for the rate of convergence of the tamed-adaptive Milstein scheme in both short and long time period. We will consider two models given in Table \ref{tab:1}. The first model is the Ginzburg-Landau model. The second model is a SDE with low regularity coefficients. We set   $X_0=x_0= 0.1$ in all the cases. It is straightforward to verify that the drift and diffusion coefficients of these equations satisfy Assumptions \textbf{A0-A2} with constants $  p_0, l_0,  \gamma, \eta, \lambda, l, \alpha$ shown in Table \ref{tab:1b}.
\begin{center}
	\begin{table}[ht]
		\begin{center}
			\renewcommand{\arraystretch}{1.3} 
	\begin{tabular}{|>{\centering\arraybackslash}p{1.5cm}|p{4cm}|p{4cm}|}
				\hline
					&
				$\mu(x)$
				&
			$\sigma(x)$
			\\
			\hline
				Model 1
				&
				$
		0.1(x-x^3)
$
				&
				$
		0.1x
$
\\
\hline
				Model 2
				&
				$
		-0.1(1+3x+x|x|^{0.5})
$
				&
				$
		 0.3(1+|x|^{1.2})
$
\\
    \hline
			\end{tabular}
		\end{center}
		\caption{\small Examples of SDEs satisfying the conditions \textbf{A0-A2}.
			\label{tab:1}}
	\end{table}
\end{center} 

\begin{center}
	\begin{table}[!ht]
		\begin{center}
			\renewcommand{\arraystretch}{1.3} 
	\begin{tabular}{|>{\centering\arraybackslash}p{1.5cm}|p{1.5cm}|p{1.5cm}|p{1.5cm}|p{1.9cm}|p{1.5cm}|p{1.5cm}|p{1.5cm}|}
				\hline
								&
    $p_0$
            &
            $l_0$
            &
            $\gamma$
            &
            $\eta$
            &
            $\lambda$
            &
            $l$
            &
            $\alpha$
     		\\
				\hline
			Model	1
				&
            12
            &
            2
            &0.65&0&0.015&$1$&$1$
				\\
    \hline
    		Model		2
				&
            6
            &
            2
            &-0.2&$ 6 \times 10^{6}$&-0.2&$ 0.3 $&0.2
				\\
    \hline
    			\end{tabular}
		\end{center}
		\caption{\small The values of constants in  conditions \textbf{A0-A2}.
			\label{tab:1b}}
	\end{table}
\end{center}
In these examples,   $p_0 \geq 4(l+\alpha+1), l_0=2 \geq \max\{2, \frac{4l}{3(1+\alpha)}\},$ hence it follows from Theorem \ref{main Thr } that the tamed-adapted Milstein scheme defined by \eqref{Milstein1} and \eqref{chooseh} converges in $L_2$-norm with a  rate of order $(1+ \alpha)/2$ over any finite time intervals. Moreover, in Example 2, since $\gamma<0, \lambda<0, $ the tamed-adapted Milstein scheme converges in $L_2$-norm with a rate of order $(1+ \alpha)/2$
on the infinite time interval $(0, +\infty)$. 
There exists a positive constant $C$ such that for any $\Delta \in (0, 1)$, it holds that 
\begin{equation*} 
	\underset{0 \le t \le T}{\sup} \bE \left[  |\widehat{X}_t - X_t| ^2\right] \le 
	C \Delta^{1+\alpha}.
	\end{equation*}
Assume that  $\widehat{X}^{(k)}_t$ and $ \widehat{X}^{(k+1)}_t$ are two approximations of the exact solution $X_t$ defined by equation \eqref{Milstein1} with  $\Delta = 2^{-k}$ and $\Delta = 2^{-(k+1)},$ respectively, where $k \geq 1.$ 
By using the inequality $|a-b|^2 \leq 2(a^2+b^2)$ valid for any $a, b \in \mathbb{R},$ we have
\begin{equation*} 
	\underset{0 \le t \le T}{\sup} \bE \left[  |\widehat{X}^{(k)}_t - \widehat{X}^{(k+1)}_t| ^2\right] \le 
	C 2^{-k(1+\alpha)},
	\end{equation*}
	for some positive constant $C.$  We approximate $\bE \left[  |\widehat{X}^{(k)}_t - \widehat{X}^{(k+1)}_t| ^2\right]$ by 
\begin{align*}
\textrm{MSE}(k) = \dfrac{1}{N_k}\sum_{m=1}^{N_k}|\widehat{X}^{(k, m)}- \widehat{X}^{(k+1, m)}|^2,
\end{align*}
where $\widehat{X}^{(k, m)}$ with $m=1,\ldots, N_k$, are $N_k$ independent copies of $\widehat{X}^{(k)}$. Note that for each $m$ and $ k,$ $\widehat{X}^{(k, m)}, \widehat{X}^{(k+1, m)}$ are generated with respect to the same Brownian path
(see Algorithm 1 in \cite{I27}).  If 
$\textrm{MSE}(k) \approx  C' 2^{-k(1+\alpha')},$
for some positive constants $\alpha'$ and $C'$, then 
$\log_2 \textrm{MSE}(k) \approx  -k(1+\alpha')+ \log_2 C'$.
As a consequence,  we can write
\begin{align} \label{model:mse}
\log_2 \textrm{MSE} (k) = -k(1+\alpha')+ \log_2 C'+ o(1).
\end{align}
Hence, the empirical strong convergence rate, $\frac{1+\alpha'}{2}$, can be estimated as the slope coefficient in the linear regression model \eqref{model:mse}. In the subsequent simulation, we set $N_k = 10^5$ for $k = 1, \ldots, 6$ and $T = 5$. The results are presented in Figure \ref{fig1}, 
where we observe that the empirical convergence rates $\frac{1+\alpha'}{2}$ are slightly higher than the theoretical rate $\frac{1+\alpha}{2}$.

\begin{figure}[htp] 
    \centering    \includegraphics[width=0.4\linewidth]{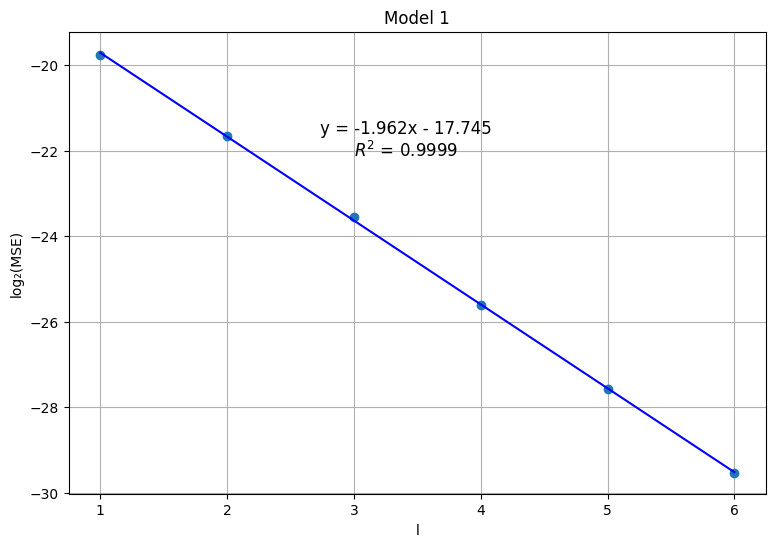}
    \includegraphics[width=0.4\linewidth]{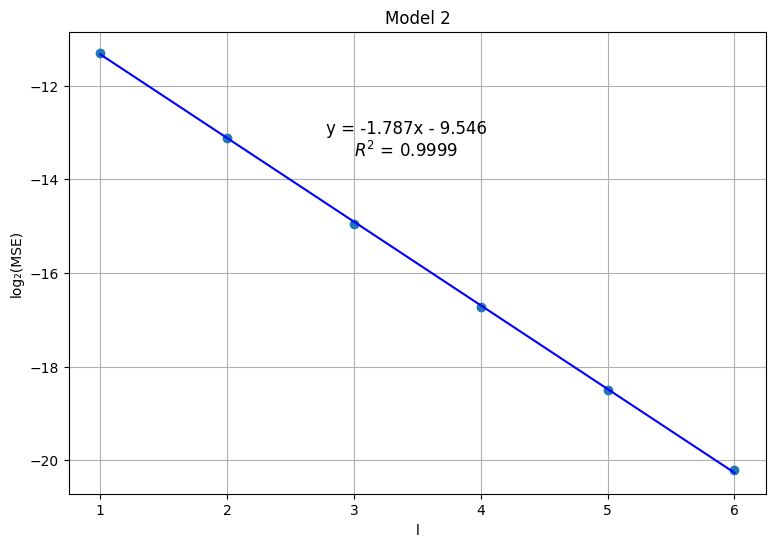}
    \caption{Values of $\log_2 \textrm{MSE}(k)$ for $T= 5$ and $1\leq k \leq 6$}
    \label{fig1}
\end{figure}

To compare the magnitude of errors with respect to $T$, Figure \ref{fig2} presents the values of  $ \log_2 \text{MSE}(k)$ for $ 1 \leq k \leq 6$ at $T = 1$, $T = 3$, $T = 5$, and $T = 10$. We observe that in Model 1, the values of $\log_2 \text{MSE}(k)$ increase significantly as $T $ grows from $1$ to $10$, while in Model 2, where we have both coefficients $\gamma$ and $\lambda$ negative, the increase in these values is much more gradual.

\begin{figure}[htp] 
    \centering    \includegraphics[width=0.4\linewidth]{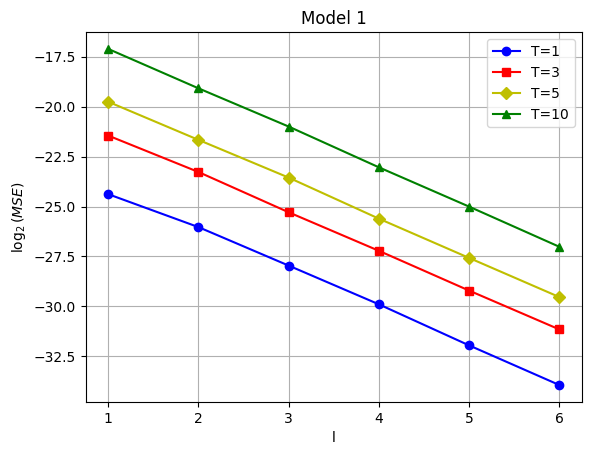}
    \includegraphics[width=0.4\linewidth]{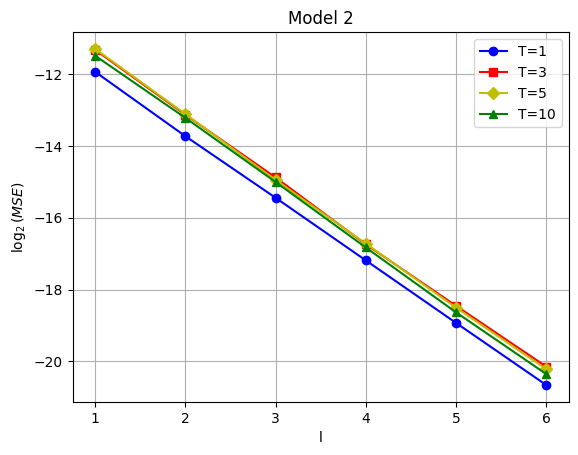}
    \caption{Values of $\log_2 \textrm{MSE}(k)$   for $T=1$, $T=3$, $T=5$ and $T= 10$, $1\leq k \leq 6$.} 
    \label{fig2}
\end{figure}

\subsection{A comparison with the tamed Milstein scheme} 
In \cite{I23}, the tamed Milstein scheme approximation $\tilde{X}$ of equation \eqref{eqn1} was defined as follows: For each fixed stepsize $\Delta \in (0,1)$, we set 
$\tilde{X}_0 = x_0$ and for any $k=0,1,\ldots$
\begin{equation} \label{def:TM}
\tilde{X}_{(k+1)\Delta} = \tilde{X}_{k\Delta} + \frac{b(\tilde{X}_{k\Delta})}{1 + \Delta |\tilde{X}_{k\Delta} |^{2}}\Delta + \frac{\sigma(\tilde{X}_{k\Delta})}{1 + \Delta |\tilde{X}_{k\Delta} |^{2}}(W_{(k+1)\Delta}-W_{k\Delta})+ \frac{1}{2} \frac{\sigma(\tilde{X}_{k\Delta})\sigma'(\tilde{X}_{k\Delta})}{1 + \Delta |\tilde{X}_{k\Delta} |^{2}}((W_{(k+1)\Delta}-W_{k\Delta})^2-\Delta).
\end{equation} 

In what follows, we compare the performance of the tamed-adaptive Milstein scheme (TAM) and the tamed Milstein scheme (TM) as the terminal time $T$ increases from 1 to 10. Due to the random nature of the timestep in the TAM scheme, we evaluate the accuracy of both methods by computing $\log _2 \mathrm{MSE}(k)$ and plotting it against $\log _2 N(T)$, where $N(T)$ denotes the average number of timesteps required to generate a single sample path.
Specifically, $N(T)$ is defined as
$$N(T) = \frac{1}{N_k} \sum_{m=1}^{N_k} N^k(m)$$ 
where $N_k$ is the total number of sample paths, and $N^k(m)$ is the number of timesteps used to simulate the $m$th sample path of the process $\left(X_t\right)_{0 \leq t \leq T}$ with $\Delta = 2^{-k}$.
For the TM scheme, which employs a fixed step size $\Delta$, the average number of timesteps is simply $N(T)=T / \Delta$.

\begin{figure}
    \centering
    \includegraphics[width=0.2\linewidth]{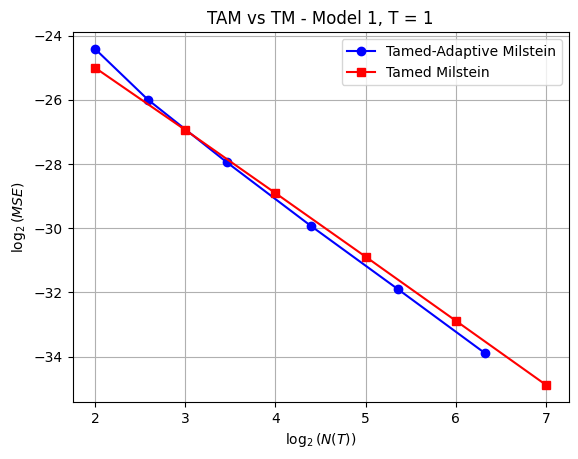}
    \includegraphics[width=0.2\linewidth]{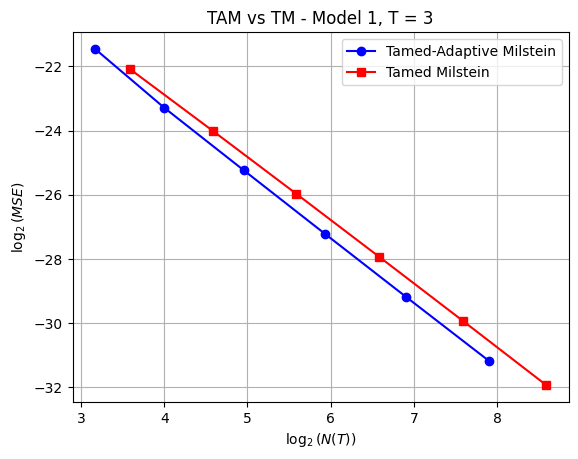}
    \includegraphics[width=0.2\linewidth]{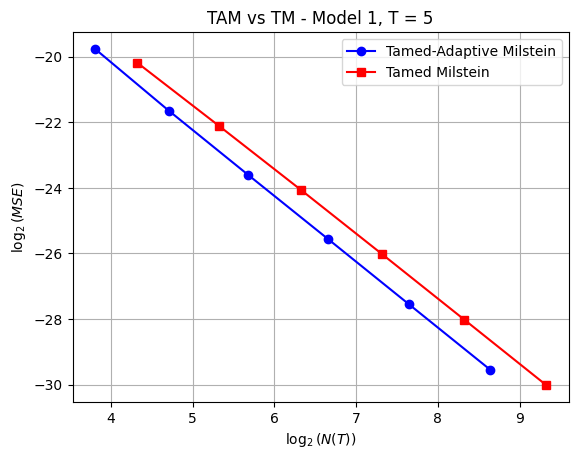}
    \includegraphics[width=0.2\linewidth]{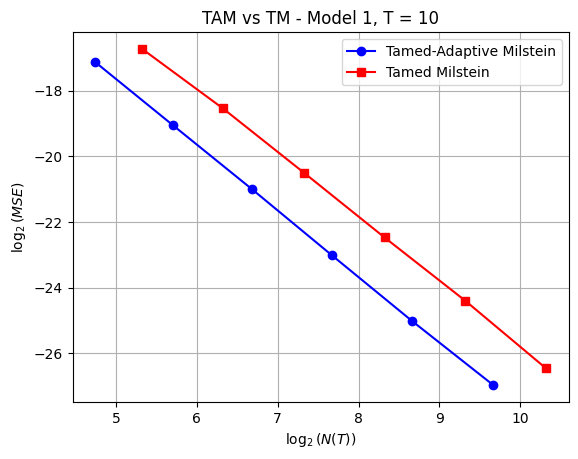}
    \caption{Comparison of TAM and TM schemes for Model 1 with $T=1, 3, 5, 10$.}
    \label{Fig:Model1}
\end{figure}
\begin{figure}
    \centering
    \includegraphics[width=0.2\linewidth]{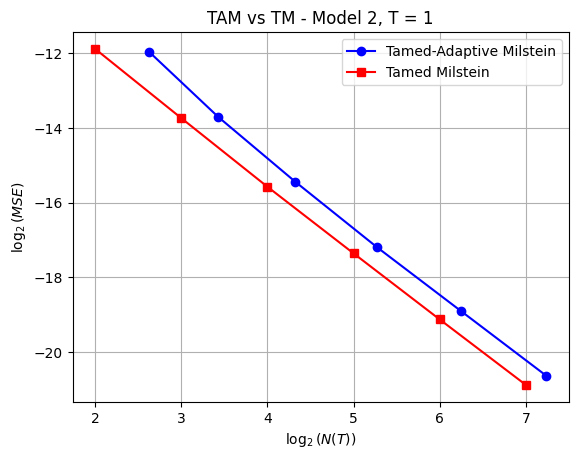}
    \includegraphics[width=0.2\linewidth]{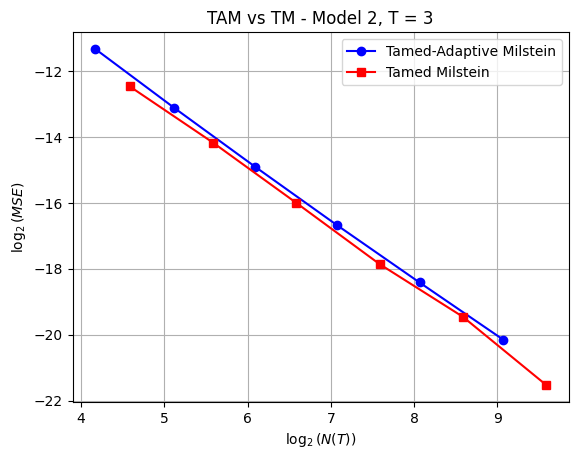}
    \includegraphics[width=0.2\linewidth]{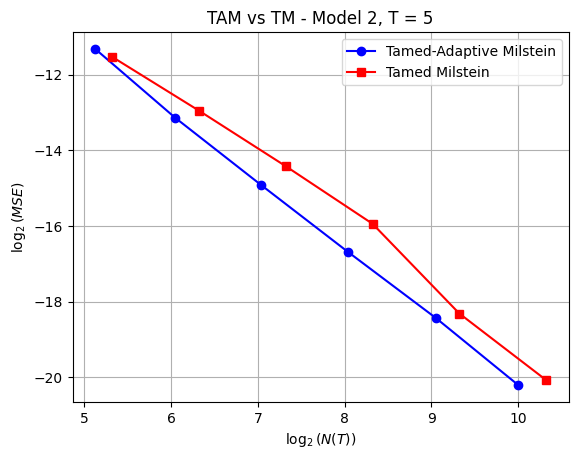}
    \includegraphics[width=0.2\linewidth]{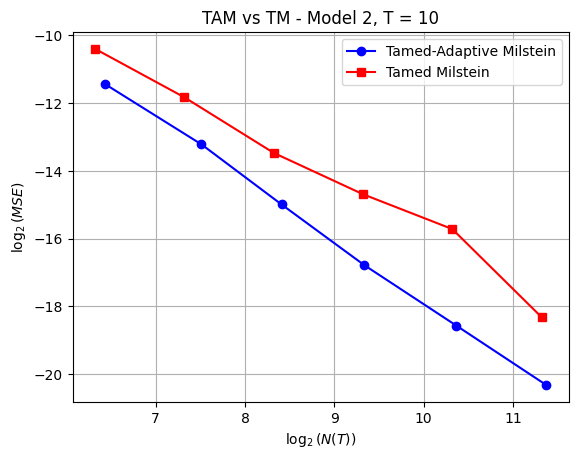}
    \caption{Comparison of TAM and TM schemes for Model 2 with $T=1, 3, 5, 10$. }
    \label{Fig:Model2}
\end{figure}
The numerical results for Model 1 and Model 2 are presented in Figure \ref{Fig:Model1} and Figure \ref{Fig:Model2}, respectively. Plots show $ \log_2(\mathrm{MSE})$ versus $\log_2(N(T))$. We have set   $N_k = 10^5$ for all $k$. From these figures, it is observed that the TM  scheme seems to perform better than the TAM  scheme when the terminal time $T$ is relatively small. However, as $T$ increases, the accuracy of the TM scheme gradually deteriorates and is eventually outperformed by the TAM scheme, even in the case of Model 1.

\section{Some auxiliary estimates} \label{sec:app}
\subsection{Estimates on the increment of the approximate solution}
In this section, we will give some estimates on the increments of the approximate solution. Using \eqref{Milstein3} and the binomial theorem, we get that for any positive integer $q$,  
	\begin{align}
	\widehat{X}_s^q
	&=\sum\limits_{0\leq i,j,r,v\leq q, i+j+r+v=q} \dfrac{q!}{i!j!r!v!}(\widehat{X}_{\underline{s}})^i\big(\mu(\widehat{X}_{\underline{s}})(s-\underline{s})\big)^j \Big(\sigma(\widehat{X}_{\underline{s}})(W_s-W_{\underline{s}})\Big)^{r} \Big(\dfrac{1}{2}q_\Delta(\widehat{X}_{\underline{s}})\left((W_s-W_{\underline{s}})^2-(s-\underline{s})\right)\Big)^{v}.
	\label{X^q1}
	\end{align}
 We will repeatedly apply the following estimate, which can be directly deduced from \eqref{chooseh} and the H\"older inequality: For any $s>0$, for any $j_1,\ldots, j_6 \in [0,1]$ such that $j_1+\ldots + j_6 \leq 1$,
\begin{align} \label{eqn:inc}
 |\mu(\hat{X}_{\underline{s}})|^{4j_1} |\mu'(\hat{X}_{\underline{s}})|^{2j_2} |\sigma(\hat{X}_{\underline{s}})|^{8j_3}
 |\sigma'(\hat{X}_{\underline{s}})|^{8j_4}  
 |q_\Delta (\hat{X}_{\underline{s}})|^{2j_5}
 |\hat{X}_{\underline{s}}|^{2l_0 j_6} 
 (s - \underline{s})\leq h_0 \Delta.
\end{align}
 
\begin{Lem} \label{Lem:-gmX^p1}
For any $p>0$, there exists a positive constant $C_p$, which does not depend on $\Delta$, such that for any $s>0$, it holds that 
	 	\begin{align}\label{-gmX^p1}
	\mathbb{E}\left[|\widehat{X}_s^p -\widehat{X}_{\underline{s}}^p|\big| \mathcal{F}_{\underline{s}}\right] \leq C_p\sum\limits_{i=0}^{p-2}  |\widehat{X}_{\underline{s}}|^i.
	\end{align}
\end{Lem}
\begin{proof}
It follows from \eqref{X^q1} that 
	\begin{align*}
	&\mathbb{E}\left[|\widehat{X}_s^p -\widehat{X}_{\underline{s}}^p|\big| \mathcal{F}_{\underline{s}}\right]\le p\vert\widehat{X}_{\underline{s}}\vert^{p-2} |\widehat{X}_{\underline{s}} \mu(\widehat{X}_{\underline{s}})|(s-\underline{s}) \\
	&\qquad+ p|\widehat{X}_{\underline{s}}|^{p-2}|\widehat{X}_{\underline{s}}\sigma(\widehat{X}_{\underline{s}})|\bE\left[|W_s-W_{\underline{s}}|\big| \mathcal{F}_{\underline{s}}\right] +\dfrac{p}{2}|\widehat{X}_{\underline{s}}^{p-2}||\widehat{X}_{\underline{s}}||q_\Delta(\widehat{X}_{\underline{s}})|\bE\left[\left|(W_s-W_{\underline{s}})^2-(s-\underline{s})\right|\big| \mathcal{F}_{\underline{s}}\right] \\
	&\qquad+\sum\limits_{0\leq i\leq p-2, i+j+r+v=p} \dfrac{ p!}{i!j!r!v!}\left|\widehat{X}_{\underline{s}}\right|^i\left|\mu(\widehat{X}_{\underline{s}})(s-\underline{s})\right|^j \\
&\qquad\qquad\times\mathbb{E}\left[\left|\sigma(\widehat{X}_{\underline{s}})(W_s-W_{\underline{s}})\right|^{r}\left|\dfrac{1}{2}q_\Delta(\widehat{X}_{\underline{s}})\left((W_s-W_{\underline{s}})^2-(s-\underline{s})\right)\right|^{v} \Big| \mf_{\underline{s}}\right].
	\end{align*}
	Using Cauchy-Schwarz inequality and \eqref{markov2}, we get
	\begin{align}\label{Holder}
&\mathbb{E}\left[\left|\sigma(\widehat{X}_{\underline{s}})(W_s-W_{\underline{s}})\right|^{r}\left|\dfrac{1}{2}q_\Delta(\widehat{X}_{\underline{s}})\left((W_s-W_{\underline{s}})^2-(s-\underline{s})\right)\right|^{v} \Big| \mf_{\underline{s}}\right] \notag\\
	&\leq \left(\mathbb{E}\left[\left|\sigma(\widehat{X}_{\underline{s}})(W_s-W_{\underline{s}})\right|^{2r} \big| \mathcal{F}_{\underline{s}}\right]\right)^{1/2}\left(\mathbb{E}\left[\left(\dfrac{1}{2}q_\Delta(\widehat{X}_{\underline{s}})\left((W_s-W_{\underline{s}})^2-(s-\underline{s})\right)\right)^{2v}\big| \mathcal{F}_{\underline{s}}\right]\right)^{1/2}\notag\\
	&\leq C_r C_v|\sm(\widehat{X}_{\underline{s}})|^r(s-\underline{s})^{r/2}\left|q_\Delta(\widehat{X}_{\underline{s}})\right|^v(s-\underline{s})^v.
	\end{align}
 This, together with \eqref{C2} and \eqref{eqn:inc}, implies \eqref{-gmX^p1}. 
\end{proof}

\begin{Lem} \label{Lem:X^p-1b1}
For any $p>0$, there exists a positive constant $C_p$, which does not depend on $\Delta$, such that for any $s>0$, it holds that 
	\begin{equation}
	\mathbb{E}\left[|\widehat{X}_s^{p-1}-\widehat{X}_{\underline{s}}^{p-1}||\mu(\widehat{X}_{\underline{s}})| \big| \mathcal{F}_{\underline{s}} \right] \leq   C_p\sum\limits_{i=0}^{p-2} |\widehat{X}_{\underline{s}}|^i. \label{X^p-1b1}
	\end{equation}
\end{Lem}

\begin{proof}
    Applying \eqref{X^q1} for $q=p-1$,  we get 
	\begin{align}\label{xb1}
	&\mathbb{E}\left[|\widehat{X}_s^{p-1}-\widehat{X}_{\underline{s}}^{p-1}||\mu(\widehat{X}_{\underline{s}})| \big| \mathcal{F}_{\underline{s}} \right]  \leq (p-1)\left|\widehat{X}_{\underline{s}}^{p-2}\right|\mu^2(\widehat{X}_{\underline{s}})(s-\underline{s})\notag\\
	&+(p-1)|\widehat{X}_{\underline{s}}^{p-2}||\mu(\widehat{X}_{\underline{s}})||\sigma(\widehat{X}_{\underline{s}})|\mathbb{E}\left[|W_s-W_{\underline{s}}|| \mathcal{F}_{\underline{s}} \right]  +\dfrac{1}{2}(p-1)|\widehat{X}_{\underline{s}}|^{p-2}|\mu(\widehat{X}_{\underline{s}})||q_\Delta(\widehat{X}_{\underline{s}})|\mathbb{E}\left[\left|(W_s-W_{\underline{s}})^2-(s-\underline{s})\right|\big| \mathcal{F}_{\underline{s}} \right]   \notag\\
&+\sum\limits_{0\leq i\leq p-3, i+j+r+v=p-1, j
\geq 1} \dfrac{(p-1)!}{i!j!r!v!}\left|\mu(\widehat{X}_{\underline{s}})\right|\left|\widehat{X}_{\underline{s}}\right|^i\left|\mu(\widehat{X}_{\underline{s}})(s-\underline{s})\right|^j \notag\\
&\qquad\times \mathbb{E}\left[\left|\sigma(\widehat{X}_{\underline{s}})(W_s-W_{\underline{s}})\right|^{r}\left|\left(\dfrac{1}{2}q_\Delta(\widehat{X}_{\underline{s}})\left((W_s-W_{\underline{s}})^2-(s-\underline{s})\right)\right)^{v}\right|\big| \mathcal{F}_{\underline{s}}\right] \notag\\
&+\sum\limits_{0\leq i\leq p-3, i+r=p-1} \dfrac{ (p-1)!}{i!r!}\left|\mu(\widehat{X}_{\underline{s}})\right|\left|\widehat{X}_{\underline{s}}\right|^i\mathbb{E}\left[\left|\sigma(\widehat{X}_{\underline{s}})(W_s-W_{\underline{s}})\right|^{r}\big| \mathcal{F}_{\underline{s}}\right]\notag\\
&+\sum\limits_{0\leq i\leq p-3, i+v=p-1} \dfrac{ (p-1)!}{i!v!}\left|\mu(\widehat{X}_{\underline{s}})\right|\left|\widehat{X}_{\underline{s}}\right|^i\mathbb{E}\left[\left|\dfrac{1}{2}q_\Delta(\widehat{X}_{\underline{s}})\left((W_s-W_{\underline{s}})^2-(s-\underline{s})\right)\right|^{v}\big| \mathcal{F}_{\underline{s}}\right]\notag\\
&+\sum\limits_{0\leq i\leq p-3, i+r+v=p-1, r, v\geq 1 } \dfrac{ (p-1)!}{i!r!v!}\left|\mu(\widehat{X}_{\underline{s}})\right|\left|\widehat{X}_{\underline{s}}\right|^i \mathbb{E}\left[\left|\sigma(\widehat{X}_{\underline{s}})(W_s-W_{\underline{s}})\right|^{r}\left|\dfrac{1}{2}q_\Delta(\widehat{X}_{\underline{s}})\left((W_s-W_{\underline{s}})^2-(s-\underline{s})\right)\right|^{v}\big| \mathcal{F}_{\underline{s}}\right].
	\end{align}
	For $0\leq i\leq p-3, i+j+r+v=p-1, j
\geq 1$, from \eqref{eqn:inc}, we have
	\begin{align*}
	\left|\mu(\widehat{X}_{\underline{s}})\right|\left|\mu(\widehat{X}_{\underline{s}})(s-\underline{s})\right|^j=\mu^2(\widehat{X}_{\underline{s}})(s-\underline{s}) \left|\mu(\widehat{X}_{\underline{s}}) (s-\underline{s})\right|^{j-1} \leq C.
	\end{align*}
	This, combined with \eqref{Holder}, yields
	\begin{align}\label{tong11}
	\left|\mu(\widehat{X}_{\underline{s}})\right|\left|\mu(\widehat{X}_{\underline{s}})(s-\underline{s})\right|^j \mathbb{E}\left[\left|\sigma(\widehat{X}_{\underline{s}})(W_s-W_{\underline{s}})\right|^{r}\left|\left(\dfrac{1}{2}q_\Delta(\widehat{X}_{\underline{s}})\left((W_s-W_{\underline{s}})^2-(s-\underline{s})\right)\right)^{v}\right|\big| \mathcal{F}_{\underline{s}}\right] \leq C.
	\end{align}
	For $0 \leq i \leq p-3, i+r=p-1,$ using H\"older's inequality, \eqref{markov2} and  \eqref{eqn:inc},
	we get
	\begin{align}\label{tong21}
	\left|\mu(\widehat{X}_{\underline{s}})\right|\mathbb{E}\left[\left|\sigma(\widehat{X}_{\underline{s}})(W_s-W_{\underline{s}})\right|^{r}\big| \mathcal{F}_{\underline{s}}\right]=\mathbb{E}\left[\left|\mu(\widehat{X}_{\underline{s}})\sigma(\widehat{X}_{\underline{s}})(W_s-W_{\underline{s}})\right|\left|\sigma(\widehat{X}_{\underline{s}})(W_s-W_{\underline{s}})\right|^{r-1}\big| \mathcal{F}_{\underline{s}}\right] \leq C.
	\end{align}
	For $0 \leq i \leq p-3, i+v=p-1,$ a similar  argument yields to
	\begin{align}\label{tong31}
	\left|\mu(\widehat{X}_{\underline{s}})\right|\mathbb{E}\left[\left|\dfrac{1}{2}q_\Delta(\widehat{X}_{\underline{s}})\left((W_s-W_{\underline{s}})^2-(s-\underline{s})\right)\right|^{v}\big| \mathcal{F}_{\underline{s}}\right] \leq C.
	\end{align}
	For $0 \leq i \leq p-3, i+r+v=p-1, r, v \geq 1,$ using H\"older's inequality, \eqref{markov2} and  \eqref{eqn:inc},
	we get
	\begin{align}\label{tong41}
	&\left|\mu(\widehat{X}_{\underline{s}})\right|\mathbb{E}\left[\left|\sigma(\widehat{X}_{\underline{s}})(W_s-W_{\underline{s}})\right|^{r}\left|\dfrac{1}{2}q_\Delta(\widehat{X}_{\underline{s}})\left((W_s-W_{\underline{s}})^2-(s-\underline{s})\right)\right|^{v}\big| \mathcal{F}_{\underline{s}}\right]\notag\\
	& =\mathbb{E}\left[\left|\mu(\widehat{X}_{\underline{s}})\sigma(\widehat{X}_{\underline{s}})(W_s-W_{\underline{s}})\right|\left|\sigma(\widehat{X}_{\underline{s}})(W_s-W_{\underline{s}})\right|^{r-1}\left|\dfrac{1}{2}q_\Delta(\widehat{X}_{\underline{s}})\left((W_s-W_{\underline{s}})^2-(s-\underline{s})\right)\right|^{v}\big| \mathcal{F}_{\underline{s}}\right]\notag\\
	&\leq \dfrac{1}{2}\mathbb{E}\left[\left|\mu(\widehat{X}_{\underline{s}})\sigma(\widehat{X}_{\underline{s}})(W_s-W_{\underline{s}})\right|^2\big| \mathcal{F}_{\underline{s}}\right]+ \dfrac{1}{2}\mathbb{E}\left[\left|\sigma(\widehat{X}_{\underline{s}})(W_s-W_{\underline{s}})\right|^{2(r-1)}\left|\dfrac{1}{2}q_\Delta(\widehat{X}_{\underline{s}})\left((W_s-W_{\underline{s}})^2-(s-\underline{s})\right)\right|^{2v}\big| \mathcal{F}_{\underline{s}}\right] \notag\\
	&\leq C.
	\end{align}
	Plugging \eqref{tong11}, \eqref{tong21}, \eqref{tong31} and \eqref{tong41} into \eqref{xb1}, we get \eqref{X^p-1b1}.
\end{proof}

\begin{Lem} \label{Lem:X^p-2sm1}
For any $p>0$, there exists a positive constant $C_p$, which does not depend on $\Delta$, such that for any $s>0$, it holds that 
	 \begin{equation}
	\mathbb{E}\left[|\widehat{X}_s^{p-2}-\widehat{X}_{\underline{s}}^{p-2}|\sigma^2(\widehat{X}_{\underline{s}}) \big| \mathcal{F}_{\underline{s}} \right]\leq  C_p\sum\limits_{i=0}^{p-2} |\widehat{X}_{\underline{s}}|^i . \label{X^p-2sm1}
	\end{equation}
\end{Lem}

\begin{proof}
 Applying \eqref{X^q1} for $q=p-2$, we get 
	\begin{align*}
	&\widehat{X}_s^{p-2} = \widehat{X}_{\underline{s}}^{p-2}  \notag\\
&+\sum\limits_{0\leq i\leq p-3, i+j+r+v=p-2} \dfrac{ (p-2)!}{i!j!r!v!}
  (\widehat{X}_{\underline{s}})^i\left(\mu(\widehat{X}_{\underline{s}})(s-\underline{s})\right)^j \left(\sigma(\widehat{X}_{\underline{s}})(W_s-W_{\underline{s}})\right)^{r}\left(\dfrac{1}{2}q_\Delta(\widehat{X}_{\underline{s}})\left((W_s-W_{\underline{s}})^2-(s-\underline{s})\right)\right)^{v}.
  \end{align*}
  Therefore, 
	\begin{align}\label{sigmax1}
	&\mathbb{E}\left[|\widehat{X}_s^{p-2}-\widehat{X}_{\underline{s}}^{p-2}|\sm^2(\widehat{X}_{\underline{s}}) \big| \mathcal{F}_{\underline{s}} \right]\notag\\ 
&\leq \sum\limits_{0\leq i\leq p-3, i+j+r+v=p-2, j
\geq 1} \dfrac{ (p-2)!}{i!j!r!v!}\sm^2(\widehat{X}_{\underline{s}})
  \left|\widehat{X}_{\underline{s}}\right|^i\left|\mu(\widehat{X}_{\underline{s}})(s-\underline{s})\right|^j \notag\\
  &\quad\times \mathbb{E}\left[\left|\sigma(\widehat{X}_{\underline{s}})(W_s-W_{\underline{s}})\right|^{r}\left|\dfrac{1}{2}q_\Delta(\widehat{X}_{\underline{s}})\left((W_s-W_{\underline{s}})^2-(s-\underline{s})\right)\right|^{v}\big| \mathcal{F}_{\underline{s}}\right] \notag\\
&\quad+\sum\limits_{0\leq i\leq p-3, i+r=p-2} \dfrac{ (p-2)!}{i!r!}\sm^2(\widehat{X}_{\underline{s}})\left|\widehat{X}_{\underline{s}}\right|^i\mathbb{E}\left[\left|\sigma(\widehat{X}_{\underline{s}})(W_s-W_{\underline{s}})\right|^{r}\big| \mathcal{F}_{\underline{s}}\right] \notag\\
&\quad+\sum\limits_{0\leq i\leq p-3, i+v=p-2} \dfrac{ (p-2)!}{i!v!}\sm^2(\widehat{X}_{\underline{s}})\left|\widehat{X}_{\underline{s}}\right|^i\mathbb{E}\left[\left|\dfrac{1}{2}q_\Delta(\widehat{X}_{\underline{s}})\left((W_s-W_{\underline{s}})^2-(s-\underline{s})\right)\right|^{v}\big| \mathcal{F}_{\underline{s}}\right] \notag\\
&\quad+\sum\limits_{0\leq i\leq p-3, i+r+v=p-2, r, v\geq 1 } \dfrac{ (p-2)!}{i!r!v!}\sm^2(\widehat{X}_{\underline{s}})\left|\widehat{X}_{\underline{s}}\right|^i \mathbb{E}\left[\left|\sigma(\widehat{X}_{\underline{s}})(W_s-W_{\underline{s}})\right|^{r}\left|\dfrac{1}{2}q_\Delta(\widehat{X}_{\underline{s}})\left((W_s-W_{\underline{s}})^2-(s-\underline{s})\right)\right|^{v}\big| \mathcal{F}_{\underline{s}}\right].
	\end{align}
For $	0\leq i\leq p-3, i+j+r+v=p-2, j
\geq 1,$ again using \eqref{Holder} and  \eqref{eqn:inc}, we obtain
\begin{align}\label{tong50}
\sm^2(\widehat{X}_{\underline{s}})
  \left|\mu(\widehat{X}_{\underline{s}})(s-\underline{s})\right|^j \mathbb{E}\left[\left|\sigma(\widehat{X}_{\underline{s}})(W_s-W_{\underline{s}})\right|^{r}\left|\dfrac{1}{2}q_\Delta(\widehat{X}_{\underline{s}})\left((W_s-W_{\underline{s}})^2-(s-\underline{s})\right)\right|^{v}\big| \mathcal{F}_{\underline{s}}\right] \leq C.
\end{align}
For $	0\leq i\leq p-3, i+r=p-2,$ using H\"older's inequality, \eqref{markov2} and \eqref{eqn:inc},  we have
	\begin{align}\label{tong51}
\mathbb{E}\left[\sm^2(\widehat{X}_{\underline{s}})\left|\sigma(\widehat{X}_{\underline{s}})(W_s-W_{\underline{s}})\right|^{r}\big| \mathcal{F}_{\underline{s}}\right]&=\mathbb{E}\left[\left|\sm^3(\widehat{X}_{\underline{s}})\left(W_s-W_{\underline{s}}\right)\right|\left|\sigma(\widehat{X}_{\underline{s}})(W_s-W_{\underline{s}})\right|^{r-1}\big| \mathcal{F}_{\underline{s}}\right] \leq C.
	 \end{align} 
	 For $0 \leq i\leq p-3, i+v=p-2,$ a similar argument yields to
	\begin{align}\label{c1}
	\sm^2(\widehat{X}_{\underline{s}})\mathbb{E}\left[\left|\dfrac{1}{2}q_\Delta(\widehat{X}_{\underline{s}})\left((W_s-W_{\underline{s}})^2-(s-\underline{s})\right)\right|^{v}\big| \mathcal{F}_{\underline{s}}\right] \leq C.
	\end{align}
	 For $0\leq i\leq p-3, i+r+v=p-2, r, v\geq 1,$ we have
	 \begin{align*}
	& \sm^2(\widehat{X}_{\underline{s}})\left|\sigma(\widehat{X}_{\underline{s}})(W_s-W_{\underline{s}})\right|^{r}\left|\dfrac{1}{2}q_\Delta(\widehat{X}_{\underline{s}})\left((W_s-W_{\underline{s}})^2-(s-\underline{s})\right)\right|^{v}\notag\\
	 &=\left|\sigma^2(\widehat{X}_{\underline{s}})(W_s-W_{\underline{s}})\right|\left|\sigma(\widehat{X}_{\underline{s}})(W_s-W_{\underline{s}})\right|^{r-1}\left|\dfrac{1}{2}\sigma(\widehat{X}_{\underline{s}})q_\Delta(\widehat{X}_{\underline{s}})\left((W_s-W_{\underline{s}})^2-(s-\underline{s})\right)\right|\\
	 &\quad\times\left|\dfrac{1}{2}q_\Delta(\widehat{X}_{\underline{s}})\left((W_s-W_{\underline{s}})^2-(s-\underline{s})\right)\right|^{v-1} \\
	 &\leq \dfrac{1}{2}\left|\sigma^2(\widehat{X}_{\underline{s}})(W_s-W_{\underline{s}})\right|^2\left|\sigma(\widehat{X}_{\underline{s}})(W_s-W_{\underline{s}})\right|^{2(r-1)}\\
	 &\quad +\dfrac{1}{8}\left|\sigma(\widehat{X}_{\underline{s}})q_\Delta(\widehat{X}_{\underline{s}})\left((W_s-W_{\underline{s}})^2-(s-\underline{s})\right)\right|^2\left|\dfrac{1}{2}q_\Delta(\widehat{X}_{\underline{s}})\left((W_s-W_{\underline{s}})^2-(s-\underline{s})\right)\right|^{2(v-1)}.
	 \end{align*}
	 Then, using H\"older's inequality, \eqref{markov2}, \eqref{eqn:inc}, and \eqref{C2}, we get
	 \begin{align}\label{smsmq1}
	 \bE\left[\sm^2(\widehat{X}_{\underline{s}})\left|\sigma(\widehat{X}_{\underline{s}})(W_s-W_{\underline{s}})\right|^{r}\left|\dfrac{1}{2}q_\Delta(\widehat{X}_{\underline{s}})\left((W_s-W_{\underline{s}})^2-(s-\underline{s})\right)\right|^{v}\big| \mathcal{F}_{\underline{s}}\right] \leq C.
	 \end{align}
	 Plugging \eqref{tong50} \eqref{tong51}, \eqref{c1} and \eqref{smsmq1} into \eqref{sigmax1}, we get \eqref{X^p-2sm1} 
\end{proof}

\begin{Lem} \label{Lem:$X^pq^21$}
For any $p>0$,  there exists a positive constant $C_p$, which does not depend on $\Delta$, such that for any $s>0$, it holds that 
  \begin{align}\label{$X^pq^21$}
  \mathbb{E}\left[|\widehat{X}_s|^{p-2}\left(2|\sm(\widehat{X}_{\underline{s}})q_{\Delta}(\widehat{X}_{\underline{s}})(W_s-W_{\underline{s}})|+ q^2_{\Delta}(\widehat{X}_{\underline{s}})(W_s-W_{\underline{s}})^2\right) \big| \mathcal{F}_{\underline{s}} \right]  \leq C_p\sum\limits_{i=0}^{p-2} |\widehat{X}_{\underline{s}}|^i .
  \end{align}
\end{Lem}

\begin{proof}
    	Applying \eqref{X^q1} with  $q=p-2, $ we get 
\begin{align*}
	&\mathbb{E}\left[|\widehat{X}_s|^{p-2}\left(2|\sm(\widehat{X}_{\underline{s}})q_{\Delta}(\widehat{X}_{\underline{s}})(W_s-W_{\underline{s}})|+ q^2_{\Delta}(\widehat{X}_{\underline{s}})(W_s-W_{\underline{s}})^2\right) \big| \mathcal{F}_{\underline{s}} \right]  \\
	&\leq |\widehat{X}_{\underline{s}}|^{p-2}\bE\left[2|\sm(\widehat{X}_{\underline{s}})q_{\Delta}(\widehat{X}_{\underline{s}})(W_s-W_{\underline{s}})|+ q^2_{\Delta}(\widehat{X}_{\underline{s}})(W_s-W_{\underline{s}})^2\big| \mathcal{F}_{\underline{s}} \right]   \notag\\
&+\sum\limits_{0\leq i\leq p-3, i+j+r+v=p-2} \dfrac{(p-2)!}{i!j!r!v!}
  \left|\widehat{X}_{\underline{s}}\right|^i\left|\mu(\widehat{X}_{\underline{s}})(s-\underline{s})\right|^j \notag\\
  & \times\mathbb{E}\left[\left(2|\sm(\widehat{X}_{\underline{s}})q_{\Delta}(\widehat{X}_{\underline{s}})(W_s-W_{\underline{s}})|+ q^2_{\Delta}(\widehat{X}_{\underline{s}})(W_s-W_{\underline{s}})^2\right)\left|\sigma(\widehat{X}_{\underline{s}})(W_s-W_{\underline{s}})\right|^{r}\left|\dfrac{1}{2}q_\Delta(\widehat{X}_{\underline{s}})\left((W_s-W_{\underline{s}})^2-(s-\underline{s})\right)\right|^{v}\big| \mathcal{F}_{\underline{s}}\right].
  \end{align*}
  Then, using Cauchy's inequality, we have
  \begin{align*}
 & \mathbb{E}\left[\left(2|\sm(\widehat{X}_{\underline{s}})q_{\Delta}(\widehat{X}_{\underline{s}})(W_s-W_{\underline{s}})|+ q^2_{\Delta}(\widehat{X}_{\underline{s}})(W_s-W_{\underline{s}})^2\right)\left|\sigma(\widehat{X}_{\underline{s}})(W_s-W_{\underline{s}})\right|^{r}\left|\dfrac{1}{2}q_\Delta(\widehat{X}_{\underline{s}})\left((W_s-W_{\underline{s}})^2-(s-\underline{s})\right)\right|^{v}\big| \mathcal{F}_{\underline{s}}\right] \notag\\
 &
 \leq \mathbb{E}\left[4\left(\sm(\widehat{X}_{\underline{s}})q_{\Delta}(\widehat{X}_{\underline{s}})(W_s-W_{\underline{s}})\right)^2+ q^4_{\Delta}(\widehat{X}_{\underline{s}})(W_s-W_{\underline{s}})^4\big| \mathcal{F}_{\underline{s}}\right]\\
 &\quad +\dfrac{1}{2}\bE\left[\left|\sigma(\widehat{X}_{\underline{s}})(W_s-W_{\underline{s}})\right|^{2r}\left|\dfrac{1}{2}q_\Delta(\widehat{X}_{\underline{s}})\left((W_s-W_{\underline{s}})^2-(s-\underline{s})\right)\right|^{2v}\Big| \mathcal{F}_{\underline{s}}\right]. 
  \end{align*}
  From \eqref{C2}, \eqref{markov2}  and \eqref{Holder}, we get \eqref{$X^pq^21$}.  
\end{proof}

\begin{Lem} \label{Lem:qh1}
For any $q>1$, there exists a positive constant $\widehat{C}_0$ which does not depend on $\Delta$ such that, for any $t>0$, 
\begin{align}
	\label{qh1}
	\bE \left[\left|\widehat{X}_{\underline{t}}\right|^q\right] \le 4 ^{q-1}\left(\bE\left[|\widehat{X}_t|^q\right] +\widehat{C}_0 \right).
	\end{align}
\end{Lem}

\begin{proof}
	Using \eqref{Milstein3}, we have for any $q>1$,
	\begin{align*}
	\bE \left[\left|\widehat{X}_{\underline{t}}\right|^q\right] 
	&\le 4^{q-1} \bigg(\bE\left[\left|\widehat{X}_t\right|^q\right] +\bE\left[\left|\mu(\widehat{X}_{\underline{t}})(t-\underline{t})\right|^q\right] +\bE \left[\left|\sm(\widehat{X}_{\underline{t}})(W_t-W_{\underline{t}})\right|^q\right] \\
	&\qquad+\bE \left[\left|\dfrac{1}{2}q_{\Delta}(\widehat{X}_{\underline{t}})\left((W_t-W_{\underline{t}})^2-(t-\underline{t})\right)\right|^q\right] \bigg).
	\end{align*}
	Using \eqref{eqn:inc} and  \eqref{C2}, we obtain \eqref{qh1}.
 \end{proof}

\begin{Lem}\label{lm5}
For any $p>0$, there exists a positive constant $C_p$ which does not depend on $\Delta$ such that, for any $t>0$, 
\begin{itemize}
\item[\textnormal{i)}] $\bE\left[\left|\widehat{X}_t- \widehat{X}_{\underline{t}}\right|^p \Big| \mf_{\underline{t}}\right] \leq C_p \Delta^{p/2},$
\item[\textnormal{ii)}]$\bE\left[\left|\widehat{X}_t- \widehat{X}_{\underline{t}}- \sm(\widehat{X}_{\underline{t}})(W_t- W_{\underline{t}})\right|^p \Big| \mf_{\underline{t}}\right] \leq C_p \Delta^{p}.$
\end{itemize}
\end{Lem}
\begin{proof}
i) For any $p \geq 1,$ observe that
\begin{align*}
|\widehat{X}_t- \widehat{X}_{\underline{t}}|^p\leq 3^{p-1}\left(|\mu(\widehat{X}_{\underline{t}})(t-\underline{t})|^p+|\sm(\widehat{X}_{\underline{t}})(W_t- W_{\underline{t}})|^p+\left|\dfrac{1}{2}q_{\Delta}(\widehat{X}_{\underline{t}})\left((W_t- W_{\underline{t}})^2-(t-\underline{t})\right)\right|^p\right).
\end{align*}
Then, using \eqref{markov2}, \eqref{C2}, and  \eqref{eqn:inc}  we get that for any $p  \geq 1$,
\begin{align*}
\bE\left[ |\widehat{X}_t- \widehat{X}_{\underline{t}}|^p| \mf_{\underline{t}}\right] &\leq 3^{p-1}\left(|\mu(\widehat{X}_{\underline{t}})(t-\underline{t})|^p + C_p|\sm(\widehat{X}_{\underline{t}})|^p(t-\underline{t})^{p/2}+C_p|q_{\Delta}(\widehat{X}_{\underline{t}})(t-\underline{t})|^p\right)  \leq C_p \Delta^{p/2},
\end{align*}
which shows \textnormal{i)} for any $p \geq 1.$
For $0<p<1$, it suffices to use H\"older's inequality. Thus, the result follows.

ii) For any $p \geq 1,$ note that
\begin{align*}
\left|\widehat{X}_t- \widehat{X}_{\underline{t}}-\sm(\widehat{X}_{\underline{t}})(W_t- W_{\underline{t}})\right|^p\leq 2^{p-1}\left(|\mu(\widehat{X}_{\underline{t}})(t-\underline{t})|^p+\Big|\dfrac{1}{2}q_{\Delta}(\widehat{X}_{\underline{t}})\left((W_t- W_{\underline{t}})^2-(t-\underline{t})\right)\Big|^p\right).
\end{align*}
Then,  using \eqref{markov2} and  \eqref{eqn:inc},   we get
\begin{align*}
\bE\left[\left|\widehat{X}_t- \widehat{X}_{\underline{t}}- \sm(\widehat{X}_{\underline{t}})(W_t- W_{\underline{t}})\right|^p | \mf_{\underline{t}}\right]& \leq 2^{p-1}\left(|\mu(\widehat{X}_{\underline{t}})(t-\underline{t})|^p + C_p|q_{\Delta}(\widehat{X}_{\underline{t}})(t-\underline{t})|^p\right) \leq C_p \Delta^p,
\end{align*}
which shows \textnormal{ii)} for any $p \geq 1.$
For $0<p<1$, using H\"older's inequality suffices. Thus, the result follows.
\end{proof}

\begin{Lem}\label{lm7}
Suppose that  $l_0\geq \frac{4l}{3(1+\alpha)}$. For any $p>0$, there exists a positive constant $C$ such that, for any $t \geq 0$,
$$
 \bE\left[ \left| \sm(\widehat{X}_t)- \sm( \widehat{X}_{\underline{t}})-q_{\Delta}(\widehat{X}_{\underline{t}})(W_t- W_{\underline{t}})\right|^p|\mf_{\underline{t}}\right] \leq   C\Delta^{\frac{(1+\alpha)p}{2}}. $$
\end{Lem}
\begin{proof}
Thanks to H\"older's inquality, it is sufficient to proof Lemma \ref{lm7} for $p>1$. Suppose that $p>1$. We have 
\begin{align}\label{smsm'q}
\left| \sm(\widehat{X}_t)- \sm( \widehat{X}_{\underline{t}})-q_{\Delta}(\widehat{X}_{\underline{t}})(W_t- W_{\underline{t}})\right|  \leq J_{1t}+J_{2t}+J_{3t},
\end{align}
where
\begin{align*}
&J_{1t}= | \sm(\widehat{X}_t)- \sm( \widehat{X}_{\underline{t}})-\sm'(\widehat{X}_{\underline{t}})(\widehat{X}_t- \widehat{X}_{\underline{t}})|, \\
&J_{2t}=|\sm'(\widehat{X}_{\underline{t}})||\widehat{X}_t- \widehat{X}_{\underline{t}}- \sm(\widehat{X}_{\underline{t}})(W_t- W_{\underline{t}})|,\\
&J_{3t}=|\sm'(\widehat{X}_{\underline{t}})\sm( \widehat{X}_{\underline{t}})-q_{\Delta}(\widehat{X}_{\underline{t}})||W_t- W_{\underline{t}}|.
\end{align*}
   Using the estimate in \eqref{C^1+alpha} and  \eqref{U6}, there exists a constant $C_1$ such that 
   \begin{align*}
   J_{1t} &\leq C_1 (1+|\widehat{X}_t|^{l}+|\widehat{X}_{\underline{t}}|^{l})|\widehat{X}_t- \widehat{X}_{\underline{t}}|^{1+\alpha}. 
\end{align*}
Using the fact that  $|\widehat{X}_t| \leq |\widehat{X}_{\underline{t}}| + |\widehat{X}_t -\widehat{X}_{\underline{t}}|$, there exists a constant $C_2$ such that 
  \begin{align}\label{J1}
   J_{1t}^p \leq  C_2 \left(|\widehat{X}_t- \widehat{X}_{\underline{t}}|^{p(1+\alpha+l)}+(1+|\widehat{X}_{\underline{t}}|^{pl})|\widehat{X}_t- \widehat{X}_{\underline{t}}|^{p(1+\alpha)}\right).
   \end{align}
   Now, the equation \eqref{Milstein3} yields to
   \begin{align*}
   &\left|\widehat{X}_t- \widehat{X}_{\underline{t}}\right|^{p(1+\alpha)}\\
   &\leq 3^{p(1+\alpha)-1}\left(|\mu(\widehat{X}_{\underline{t}})(t-\underline{t})|^{p(1+\alpha)}+|\sm(\widehat{X}_{\underline{t}})(W_t- W_{\underline{t}})|^{p(1+\alpha)}+|q_{\Delta}(\widehat{X}_{\underline{t}})((W_t- W_{\underline{t}})^2- (t-\underline{t}))|^{p(1+\alpha)}\right).
   \end{align*}
   Note that $|\widehat{X}_{\underline{t}}| \leq \Big( \frac{h_0\Delta}{t-\underline{t}}\Big)^{1/2l_0}$, $|\mu(\widehat{X}_{\underline{t}})| \leq \Big( \frac{h_0 \Delta}{t-\underline{t}}\Big)^{1/4}$, and  $|\sigma(\widehat{X}_{\underline{t}})| \leq \Big( \frac{h_0 \Delta}{t-\underline{t}}\Big)^{1/8}$. Using these estimates and  \eqref{markov2}, we get 
 \begin{align}\label{J11}
\bE\left[\left|\widehat{X}_t- \widehat{X}_{\underline{t}}\right|^{p(1+\alpha)}(1+|\widehat{X}_{\underline{t}}|^{pl}) \Big| \mf_{\underline{t}}\right] \leq C_3 \Delta^{\frac{(1+\alpha)p}{2}},
\end{align}
for some positive constant $C_3$, provided that $l_0 \geq \frac{4l}{3(1+\alpha)}$. 
Therefore, combining \eqref{J1}, \eqref{J11} and  Lemma \ref{lm5}, we have 
\begin{align}\label{EJ1}
\bE[J_{1t}^p|\mf_{\underline{t}}] \leq  C_4 \left( \Delta^{\frac{(1+\alpha+l)p}{2}}+\Delta^{\frac{(1+\alpha)p}{2}}\right)\leq  2C_4 \Delta^{\frac{(1+\alpha)p}{2}},
\end{align}
for some positive constant $C_4$.
Next,  the equation \eqref{Milstein3} yields to
\begin{align*}
J_{2t}^p \leq 2^{p-1}\left(\left|\sm'(\widehat{X}_{\underline{t}})\mu(\widehat{X}_{\underline{t}})(t- \underline{t})\right|^p+\left|\sm{'}(\widehat{X}_{\underline{t}})q_{\Delta}(\widehat{X}_{\underline{t}})\left((W_t- W_{\underline{t}})^2-(t- \underline{t})\right)\right|^p\right).
\end{align*}
Then, applying \eqref{markov2} and  $\max\{ |\sm'(\widehat{X}_{\underline{t}})\mu(\widehat{X}_{\underline{t}})(t- \underline{t})|, |\sm'(\widehat{X}_{\underline{t}})q_{\Delta}(\widehat{X}_{\underline{t}})(t- \underline{t})|\} \leq C_0\Delta,$ for some positive constant $C_0$, there exists a constant $C_5$ such that, 
\begin{align}\label{EJ2}
\sup_{t\geq0} \bE[J_{2t}^p|\mf_{\underline{t}}] \leq C_5 \Delta^p.
\end{align}
Now, \eqref{U8} yields that
$J_{3t}^p \leq L^p \Delta^{p/2}\left|\sm'(\widehat{X}_{\underline{t}})\sm( \widehat{X}_{\underline{t}})\right|^{2p}\left|W_t- W_{\underline{t}}\right|^p.$
This, combined with \eqref{markov2} and the fact that $\max\lbrace |\sm'(\widehat{X}_{\underline{t}})\sm( \widehat{X}_{\underline{t}})|^4(t- \underline{t})\rbrace \leq C_0\Delta$, implies that
\begin{align}\label{EJ3}
\sup_{t \geq0} \bE[J_{3t}^p|\mf_{\underline{t}}] \leq C_6 \Delta^p,
\end{align}
for some positive constant $C_6$. Consequently, putting \eqref{smsm'q}, \eqref{EJ1}, \eqref{EJ2} and \eqref{EJ3} together, we obtain that 
   for any $p\geq 1$ and $l_0\geq \frac{4l}{3(1+\alpha)},$ there exists a positive $C$ such that 
\begin{align*}
 \sup_{t \geq 0} \bE\left[ \left| \sm(\widehat{X}_t)- \sm( \widehat{X}_{\underline{t}})-q_{\Delta}(\widehat{X}_{\underline{t}})(W_t- W_{\underline{t}})\right|^p|\mf_{\underline{t}}\right] 
 & \leq  C\Delta^{\frac{(1+\alpha)p}{2}}\textcolor{black}{,}
 \end{align*}
 which finishes the proof. 
 \end{proof}

\subsection{Estimates on the increment of the exact solution}
In this section, we will give some estimates on the difference between $X_t$ and $X_{\underline{t}}$. 
\begin{Lem}\label{lm1}
For any $p \in [2, p_0]$, there exists a positive constant $C= C(p, \gamma, \eta)$ such that for any $t \geq 0,$
\begin{align*} 
\bE[\vert X_t\vert ^p| \mf_{\underline{t}}] \leq |X_{\underline{t}}|^p+ C(t- \underline{t})\left(1+|X_{\underline{t}}|^p\right).
\end{align*} 
\end{Lem}
\begin{proof}
Recall that for $t>0$,
\begin{align}\label{dynamic eq}
X_{t} = X_{\underline{t}}+ \int_{\underline{t}}^t \mu(X_s)ds+\int_{\underline{t}}^t \sm(X_s)dW_s.
\end{align}
For each $N>0$, let $\tau_N = \inf\{t \geq 0: |X_t| \geq N\}.$ It is clear that $\tau_N \uparrow +\infty$ as $N \to \infty$.  Using It\^o's formula, we have that for any $a \in \mathbb{R}$, 
\begin{align*}
e^{-a(t \wedge \tau_N)}|X_{t\wedge \tau_N}|^p&=e^{-a(\underline{t} \wedge \tau_N)} |X_{\underline{t} \wedge \tau_N}|^p+\int_{\underline{t}\wedge \tau_N}^{t\wedge \tau_N} e^{-as}\left[-a|X_s|^p + p|X_s|^{p-2}\Big(X_s \mu(X_s)+ \dfrac{p-1}{2}\sm^2(X_s)\Big)\right]ds \\
&\quad+ p \int_{\underline{t}\wedge \tau_N}^{t\wedge \tau_N} e^{-as}|X_s|^{p-2} X_s \sm(X_s)dW_s.
\end{align*}
Thanks to \textbf{A1}, we have 
\begin{align*}
-a|X_s|^p + p|X_s|^{p-2}\Big(X_s \mu(X_s)+ \dfrac{p-1}{2}\sm^2(X_s)\Big) \leq  (-a+ \gamma p)|X_s|^p + \eta p|X_s|^{p-2}.
\end{align*}
Using Young's inequality, $|X_s|^{p-2} \leq \frac{p-2}{p}|X_s|^p+\frac{2}{p}$. Set   $a=\gamma p+(p-2) \eta$, we get 
$(-a+ \gamma p)|X_s|^p + \eta p|X_s|^{p-2} \leq  2\eta.$ Furthermore, 
$ \int_{\underline{t}\wedge \tau_N}^{t\wedge \tau_N} e^{-as} ds = \int_{\underline{t}}^{t} e^{-at}e^{-a(s-t)} 1_{\{ s \leq \tau_N\}}ds \leq e^{-at} e^{h_0 |a|\Delta}(t - \underline{t}).$
Therefore, 
\begin{align*}
e^{-a(t\wedge \tau_N)}|X_{t\wedge \tau_N}|^p &\leq e^{-a(\underline{t}\wedge \tau_N)} |X_{\underline{t}\wedge \tau_N}|^p
+2\eta e^{-at} e^{h_0 |a|\Delta}(t - \underline{t})
+p\int_{\underline{t}}^{t} e^{-as}|X_s|^{p-2} X_s \sm(X_s) 1_{\{s \leq \tau_N\}}dW_s.
\end{align*}
Taking the conditional expectation of both sides with respect to $\mf_{\underline{t}}$, we get 
\begin{align*}
\bE \left[ e^{-a(t\wedge \tau_N)}|X_{t\wedge \tau_N}|^p | \mf_{\underline{t}} \right] \leq e^{-a(\underline{t}\wedge \tau_N)} |X_{\underline{t}\wedge \tau_N}|^p
+2\eta e^{-at} e^{h_0 |a|\Delta}(t - \underline{t}).
\end{align*}
Let $N \to \infty$ and apply Fatou's lemma, we get 
\begin{align*}
e^{-at}\bE \left[ |X_{t\wedge \tau_N}|^p | \mf_{\underline{t}} \right] \leq e^{-a\underline{t}} |X_{\underline{t}}|^p
+2\eta e^{-at} e^{h_0 |a|\Delta}(t - \underline{t}).
\end{align*}
Now, using the estimate $0\leq  t - \underline{t} \leq  h_0\Delta$ and the inequality  $e^{x} \leq 1+x e^{x}$ valid for all $x \in \bR$, we get 
\begin{align*}
\bE\left[|X_t|^p| \mf_{\underline{t}}\right] &\leq (1 + |a|(t - \underline{t})e^{h_0|a|\Delta})|X_{\underline{t}}|^p + 2|\eta|e^{h_0 |a|\Delta} (t - \underline{t})\\
& \leq |X_{\underline{t}}|^p + (|a|+2|\eta|)e^{h_0|a|} (t-\underline{t})(|X_{\underline{t}}|^p+ 1),
\end{align*}
which is the desired result.
\end{proof}

\begin{Lem}\label{lm2}
For any $p \in [2, \frac{p_0}{l + \alpha+1}]$, there exists a positive constant $C$   such that for any $ t \geq 0,$
$$
\left\vert \bE[ X_t - X_{\underline{t}}|\mf_{\underline{t}}]\right\vert^p \leq  C(t- \underline{t})^p\left(1+|X_{\underline{t}}|^{p(l + \alpha+1)}\right). $$
\end{Lem}
\begin{proof}
Using \eqref{dynamic eq}, Jensen's inequality and  \eqref{U4},  we get that for any $p \geq 2,$
\begin{align*}
\left|\bE[X_t- X_{\underline{t}}| \mf_{\underline{t}}]\right|^p &=\left|\bE\left[\int_{\underline{t}}^t \mu(X_s)ds| \mf_{\underline{t}}\right]\right|^p \leq \bE \left[\left| \int_{\underline{t}}^t \mu(X_s)ds \right| ^p \Big| \mf_{\underline{t}}\right]\\
& \leq (t-\underline{t})^{p-1}\bE\left[\int_{\underline{t}}^t |\mu(X_s)|^p ds| \mf_{\underline{t}}\right]\\
&\leq 2^{p-1}L^p(t-\underline{t})^{p-1}\int_{\underline{t}}^t \left(1+\bE\left[\left|X_s\right|^{p(l + \alpha+1)}| \mf_{\underline{s}}\right]\right)ds.
\end{align*}
Thus, Lemma  \ref{lm1} deduces that for any $p \in [2, \frac{p_0}{l + \alpha+1}]$, there exist positive constants $C_1, C_2$ such that 
\begin{align*}
\left|\bE[X_t- X_{\underline{t}}| \mf_{\underline{t}}]\right|^p& \leq 2^{p-1}L^p (t-\underline{t})^{p-1}\int_{\underline{t}}^t \left(1+\left|X_{\underline{s}}\right|^{p(l + \alpha+1)}+C_1 (s-\underline{s})(1+\left|X_{\underline{s}}\right|^{p(l + \alpha+1)})\right)ds\\
& \leq 2^{p-1}L^p(t-\underline{t})^{p-1}\left((t-\underline{t})(1+\left|X_{\underline{t}}\right|^{p(l + \alpha+1)})+\dfrac{1}{2}C_1(t-\underline{t})^2(1+\left|X_{\underline{t}}\right|^{p(l + \alpha+1)})\right)\\
&\leq C_2 (t-\underline{t})^p\left(1+\left|X_{\underline{t}}\right|^{p(l + \alpha+1)}\right).
\end{align*}
Thus, the result follows.
\end{proof}
\begin{Lem}\label{lm3}
For any $p\in[2, \frac{p_0}{l + \alpha+1}]$, there exists a positive constant $C$ such that, for any $t \geq 0,$ 
\begin{align*}
 \bE[ \vert X_t - X_{\underline{t}}\vert^p|\mf_{\underline{t}}]\vert \leq  
 C(t- \underline{t})^{p/2}
 \left(1+|X_{\underline{t}}|^{p(l + \alpha+1)}\right). 
 \end{align*}
\end{Lem}
\begin{proof}
The equation \eqref{dynamic eq} yields that for any $ p \geq 2$,
\begin{align}\label{dix}
\left| X_t - X_{\underline{t}}\right|^p \leq 2^{p-1}\left( \left|\int_{\underline{t}}^t \mu(X_s)ds \right| ^p+\left|\int_{\underline{t}}^t \sm(X_s)dW_s \right| ^p\right).
\end{align} 
Since $p < \frac{p_0}{l + \alpha+1}$,  using the argument in the proof of Lemma \ref{lm2}, we have 
\begin{align}\label{Eb1}
 \bE \left[\left| \int_{\underline{t}}^t \mu(X_s)ds \right| ^p| \mf_{\underline{t}}\right] \leq  C(t-\underline{t})^p\left(1+\left|X_{\underline{t}}\right|^{p(l + \alpha+1)}\right).
\end{align}
Moreover, using Burkholder's inequality, there exists a positive constant $c_p$ such that 
\begin{align*}
 \bE \left[\left| \int_{\underline{t}}^t \sigma(X_s)dW_s \right| ^p| \mf_{\underline{t}}\right] \leq  c_p (t-\underline{t})^{p/2-1} \int_{\underline{t}}^t \mathbb{E}\left[|\sigma(X_s)|^p |  \mf_{\underline{t}}\right] ds. 
\end{align*}
If $\underline{t} < s < t$, then $\underline{s} = \underline{t}$. Thanks to \eqref{U4} and Lemma \ref{lm1}, there exists a positive constant $C_1$ such that 
\begin{align*}
    \mathbb{E}\left[|\sigma(X_s)|^p |  \mf_{\underline{t}}\right] 
    &\leq L^p 2^{p-1} \left(1+ \bE\left[|X_s|^{p(l + \alpha + 1)}| \mf_{\underline{s}}\right] \right)\\
    &\leq  L^p 2^{p-1} \left(1+ |X_{\underline{t}}|^{p(l + \alpha + 1)} + C_1 (s - \underline{t})(1 + |X_{\underline{t}}|^{p(l + \alpha + 1)} )\right).
\end{align*}
Therefore, there exists a positive constant $C_2$ such that 
\begin{align*}
 \bE \left[\left| \int_{\underline{t}}^t \sigma(X_s)dW_s \right| ^p| \mf_{\underline{t}}\right] 
& \leq c_p (t-\underline{t})^{p/2-1} \int_{\underline{t}}^t \left(1+ |X_{\underline{t}}|^{p(l + \alpha + 1)} + C_1 (s - \underline{t})(1 + |X_{\underline{t}}|^{p(l + \alpha + 1)} )\right) ds\\
 & \leq C_2 (t-\underline{t})^{p/2}(1+ |X_{\underline{t}}|^{p(l + \alpha + 1)}).
\end{align*}
This fact and estimates \eqref{dix} and \eqref{Eb1} imply the desired result.
 \end{proof}

\begin{Lem}\label{lm4}
For any $p\in [2,\frac{p_0}{2(l + \alpha+1)}]$, there exists a positive constant $C$ such that. for any $t \geq 0,$
$$
 \bE\left[ \left| X_t - X_{\underline{t}}-\sm(X_{\underline{t}})(W_t- W_{\underline{t}})\right|^p|\mf_{\underline{t}}\right] \leq  C(t- \underline{t})^{p}\left(1+|X_{\underline{t}}|^{2p(l + \alpha+1)}\right). $$
\end{Lem}
\begin{proof}
We proceed as in the proof of Lemma \ref{lm3}. Observe that the equation \eqref{dynamic eq} yields that for any $p \geq 2,$ 
\begin{align}\label{diffx}
\left| X_t - X_{\underline{t}}-\sm(X_{\underline{t}})(W_t- W_{\underline{t}})\right|^p \leq 2^{p-1}\left( \left|\int_{\underline{t}}^t \mu(X_s)ds \right| ^p+\left|\int_{\underline{t}}^t (\sm(X_s)-\sm(X_{\underline{s}}))dW_s \right| ^p \right).
\end{align} 
Using Burkholder's inequality, there exists a positive constant $c_p$ such that 
\begin{align*}
 \bE \left[\left| \int_{\underline{t}}^t (\sigma(X_s) - \sigma(X_{\underline{s}}))dW_s \right| ^p| \mf_{\underline{t}}\right] \leq  c_p (t-\underline{t})^{p/2-1} \int_{\underline{t}}^t \mathbb{E}\left[|\sigma(X_s) -  \sigma(X_{\underline{s}})|^p |  \mf_{\underline{t}}\right] ds. 
\end{align*}
Thanks to \eqref{U2}, if $\underline{t} < s < t$,
\begin{align*}
&    \mathbb{E}\left[|\sigma(X_s) -  \sigma(X_{\underline{s}})|^p |  \mf_{\underline{t}}\right] \leq L^{p} \mathbb{E}\left[(1+ |X_s|^{l + \alpha} + |X_{\underline{s}}|^{l + \alpha})^p |X_s - X_{\underline{s}}|^p|\mf_{\underline{t}}\right]\\
  & \leq 3^{p-1}L^p  (s - \underline{s})^{p/2}   \mathbb{E}\left[1+ |X_s|^{p(l + \alpha)} + |X_{\underline{s}}|^{p(l + \alpha)}  |\mf_{\underline{t}}\right] 
  +  3^{p-1}L^p  (s - \underline{s})^{-p/2}   \mathbb{E}\left[|X_s- X_{\underline{s}}|^{2p}  |\mf_{\underline{t}}\right].
\end{align*}
It follows from Lemma \ref{lm1} and Lemma \ref{lm3} that there exists a positive constant $C_1$ such that 
\begin{align*}
\mathbb{E}\left[|\sigma(X_s) -  \sigma(X_{\underline{s}})|^p |  \mf_{\underline{t}}\right] \leq C_1 (s- \underline{s})^{p/2}(1 + |X_{\underline{s}}|^{2p(l + \alpha+1)}).
\end{align*}
This fact and estimates \eqref{Eb1} and \eqref{diffx} imply the desired result. 
\end{proof}

\begin{Lem} \label{ulX_ut}
For any $p \in (0, p_0]$, and $T>0$, there exists a finite positive constant $K$ such that, 
$$\sup_{t\geq 0} \mathbb{E}[ |X_{\underline{t}}|^p] \leq K.$$
Moreover, when $\gamma<0$, constant $K$ does not depend on $T$.
\end{Lem}  

\begin{proof}
Thanks to H\"older's inequality, it is sufficient to prove Lemma \ref{ulX_ut} for $p\in [2,p_0]$. Suppose $p\geq 2$.
For each $N>0$, let $\tau_N = \inf\{t \geq 0: |X_t| \geq N\}$.
It follows from It\^o's formula that 
\begin{align*}
e^{-\gamma (\underline{t}\wedge \tau_N)} |X_{\underline{t}  \wedge \tau_N}|^ p 
=& |x_0|^p +  \int_0^{\underline{t}  \wedge \tau_N} pe^{-\gamma s} |X_s|^{p-2} \Big( X_s \mu(X_s) + \frac 12 (p-1)\sigma^2(X_s) - \gamma|X_s|^2\Big)ds \\
&+ \int_0^{\underline{t}  \wedge \tau_N} p |X_s|^{p-2} e^{-\gamma s}  X_s \sigma(X_s) dW_s.
\end{align*}
Taking expectation of both sides and applying  \textbf{A1}, we get  
\begin{align*}
\bE \left[ e^{-\gamma (\underline{t}\wedge \tau_N)} |X_{\underline{t}  \wedge \tau_N}|^ p \right] 
\leq & |x_0|^p +  \bE \left[ \int_0^{\underline{t}  \wedge \tau_N} \eta p e^{-\gamma s} |X_s|^{p-2} ds\right] \\
\leq & |x_0|^p + \int_0^t \eta p e^{-\gamma s} \bE \left[  |X_s|^{p-2}\right] ds. 
\end{align*}
Let $N \to \infty$ and apply Fatou's lemma, we get 
\begin{align*}
\bE \left[ e^{-\gamma \underline{t}} |X_{\underline{t} }|^ p \right] \leq & |x_0|^p + \int_0^t \eta p e^{-\gamma s} \bE \left[  |X_s|^{p-2}\right] ds,
\end{align*}
which implies that 
\begin{align*}
\bE \left[|X_{\underline{t}  }|^ p \right] \leq & e^{\gamma t} e^{h_0 |\gamma| \Delta} \left( |x_0|^p + \int_0^t \eta p e^{-\gamma s} \bE \left[  |X_s|^{p-2}\right] ds\right).
\end{align*}
Using Proposition \ref{moment nghiem dung}, we get the desired result. 

\end{proof}

\begin{Lem}\label{lm8}
For any $p_0\geq 4 (l+\alpha+1)$, there exists a positive constant $C$ such that, for any $ t \geq 0,$
$$
 \left| \bE\left[ (X_t - X_{\underline{t}}- \widehat{X}_t+ \widehat{X}_{\underline{t}})(W_t- W_{\underline{t}})|\mf_{\underline{t}}\right] \right|
 \leq  C(t- \underline{t})^{\frac{3+\alpha}{2}}(1+|X_{\underline{t}}|^{l + (1+\alpha)(l + \alpha + 1) })+|\sm(X_{\underline{t}})-\sm(\widehat{X}_{\underline{t}})|(t-\underline{t}). $$
\end{Lem}
\begin{proof}
	The expressions \eqref{Milstein3}, \eqref{dynamic eq} and  $\int_{\underline{t}}^t(W_s- W_{\underline{s}})dW_s= \frac{1}{2}((W_t- W_{\underline{t}})^2-(t-\underline{t}))$ yield that
\begin{align}\label{Diff}
&X_t - X_{\underline{t}}- \widehat{X}_t+ \widehat{X}_{\underline{t}}= I_{1t}+I_{2t}+I_{3t}+I_{4t}+I_{5t}+I_{6t},
\end{align}
where
\begin{align*}
 I_{1t}& =\int_{\underline{t}}^t\left(\mu(X_s)- \mu(X_{\underline{s}})\right)ds, \quad I_{2t} = \left(\mu(X_{\underline{t}})-\mu(\widehat{X}_{\underline{t}})\right)(t-\underline{t}),\\
I_{3t} &= \int_{\underline{t}}^t\left(\sm(X_s)- \sm(X_{\underline{s}})-\sm'(X_{\underline{s}})(X_s- X_{\underline{s}}) \right)dW_s,\quad I_{4t} = \int_{\underline{t}}^t\sm'(X_{\underline{s}})\left(X_s- X_{\underline{s}}-\sm(X_{\underline{s}})(W_s- W_{\underline{s}})\right)dW_s,\\
I_{5t}& = \dfrac{1}{2}\left(\sm(X_{\underline{t}})\sm'(X_{\underline{t}})-q_{\Delta}(\widehat{X}_{\underline{t}})\right)\left((W_t- W_{\underline{t}})^2-(t-\underline{t})\right),\quad I_{6t}= \left(\sm(X_{\underline{t}})- \sm(\widehat{X}_{\underline{t}})\right)(W_t- W_{\underline{t}}).
\end{align*}
First, observe that
\begin{align}\label{I25}
\bE\left[I_{2t}(W_t- W_{\underline{t}})| \mf_{\underline{t}}\right]= \bE\left[I_{5t}(W_t- W_{\underline{t}})| \mf_{\underline{t}}\right]=0.
\end{align}
It follows from Cauchy-Schwart's inequality that 
\begin{align*}
\left| \mathbb{E} \left[ I_{1t}(W_t - W_{\underline{t}}) |\mf_{\underline{t}}\right] \right|^2  
& \leq \bE \left[  I_{1t}^2| \mf_{\underline{t}}\right]   \bE \left[  (W_t - W_{\underline{t}})^2| \mf_{\underline{t}}\right]  \\
& \leq (t - \underline{t})^2  \int_{\underline{t}}^t \bE \left[ ( (\mu(X_s) - \mu(X_{\underline{s}}))^2| \mf_{\underline{t}}\right] ds. 
\end{align*}
Thanks to \eqref{U2}, 
\begin{align*}
\left| \mathbb{E} \left[ I_{1t}(W_t - W_{\underline{t}}) |\mf_{\underline{t}}\right] \right|^2  
& \leq 3L^2 (t - \underline{t})^2  \int_{\underline{t}}^t \bE \left[ ( X_s - X_{\underline{s}})^2 (1 + |X_s|^{2(l + \alpha)} + |X_{\underline{s}}|^{2(l + \alpha)}) | \mf_{\underline{s}}\right] ds\\
& \leq 3L^2 \sqrt{3} (t - \underline{t})^2  \int_{\underline{t}}^t  
\sqrt{\bE \left[ ( X_s - X_{\underline{s}})^4 | \mf_{\underline{s}} \right]}
\sqrt{\bE \left[1 + |X_s|^{4(l + \alpha)} + |X_{\underline{s}}|^{4(l + \alpha)} | \mf_{\underline{s}}\right] }ds. 
\end{align*}
Since $4(l+\alpha+1) \leq p_0$, using apply Lemma \ref{lm1} and Lemma \ref{lm3}, there exists a positive constant $C_1$ such that, for any $t\geq 0$, 
\begin{align*}
\left| \mathbb{E} \left[ I_{1t}(W_t - W_{\underline{t}}) |\mf_{\underline{t}}\right] \right|^2  
&\leq C_1  (t - \underline{t})^4 \sqrt{(1+|X_{\underline{t}}|^{4(l + \alpha + 1)})(1 + |X_{\underline{t}}|^{4(l + \alpha )})},
\end{align*}
which implies that 
\begin{align} \label{I1W}
\left| \mathbb{E} \left[ I_{1t}(W_t - W_{\underline{t}}) |\mf_{\underline{t}}\right] \right| 
&\leq \sqrt{2C_1}  (t - \underline{t})^2 (1 + |X_{\underline{t}}|^{2(l + \alpha) + 1}).
\end{align}
Now, note that
\begin{align*}
\bE[I_{3t}(W_t- W_{\underline{t}}) | \mf_{\underline{t}}] =\bE\left[\int_{\underline{t}}^t \left(\sm(X_s)- \sm(X_{\underline{s}})-\sm'(X_{\underline{s}})(X_s- X_{\underline{s}}) \right)ds | \mf_{\underline{t}}\right]. 
\end{align*}
Then, using the argument in \eqref{C^1+alpha} and  \eqref{U6}, there exists a positive constant $C_2$ such that, for all $s\geq0$, 
\begin{align*}
|\sm(X_s)- \sm(X_{\underline{s}})-\sm'(X_{\underline{s}})(X_s- X_{\underline{s}})| & \leq C_2 (1+|X_s|^{l}+|X_{\underline{s}}|^{l}) |X_s- X_{\underline{s}}|^{1+\alpha}. 
\end{align*}
Thus, 
\begin{align*}
|\bE[I_{3t}(W_t- W_{\underline{t}}) | \mf_{\underline{t}}]| 
&\leq  \sqrt{3} C_2 \int_{\underline{t}}^t  \sqrt{\bE \left[ 1 + |X_s|^{2l} + |X_{\underline{s}}|^{2l}  | \mf_{\underline{s}}\right] } \sqrt{\bE \left[ |X_s - X_{\underline{s}}|^{2(1+ \alpha)}  | \mf_{\underline{s}}\right] }ds.
\end{align*}
Since $2(1+\alpha)(l+\alpha+1) \leq 4 (l+\alpha+1) \leq p_0$ and $2l \leq p_0$, using Lemma \ref{lm1} and Lemma \ref{lm3}, there exists a positive constant $C_3$ such that, for any $t\geq0$, 
\begin{align} \label{I3W}
|\bE[I_{3t}(W_t- W_{\underline{t}}) | \mf_{\underline{t}}]| \leq C_3  (t - \underline{t})^{(3+\alpha)/2} (1 + |X_{\underline{t}}|^{l+(1 + \alpha)(l+\alpha+1)}).
\end{align}
Next, we have 
\begin{align*}
&\bE[I_{4t}(W_t- W_{\underline{t}}) | \mf_{\underline{t}}] =\bE\left[\int_{\underline{t}}^t\sm'(X_{\underline{s}})\left(X_s- X_{\underline{s}}-\sm(X_{\underline{s}})(W_s- W_{\underline{s}})\right)ds| \mf_{\underline{t}}\right] \\
&= \bE\left[ \int_{\underline{t}}^t\sm'(X_{\underline{s}})
\left( \int^s_{\underline{s}} \mu(X_u)du 
+ \int^s_{\underline{s}}  (\sigma(X_u) - \sigma(X_{\underline{s}}))dW_u \right)ds  | \mf_{\underline{t}}\right]  = \sm'(X_{\underline{t}})  \int_{\underline{t}}^t \int^s_{\underline{s}} \bE[ \mu(X_u) | \mf_{\underline{t}}] du ds\textcolor{black}{.} 
\end{align*}
Thus, it follows from \eqref{U4} and \eqref{U6} that there exists a constant $C_4$ such that, for any $t\geq0$, 
\begin{align*}
&|\bE[I_{4t}(W_t- W_{\underline{t}}) | \mf_{\underline{t}}] | \leq C_4 (1+ |X_{\underline{t}}|^{l + \alpha})  \int_{\underline{t}}^t \int^s_{\underline{s}} \bE[ 1+ |X_u|^{l + \alpha + 1} | \mf_{\underline{t}}] du ds. 
\end{align*}
Since $l + \alpha + 1\leq p_0$, using Lemma \ref{lm1}, there exists a positive constant $C_5$ such that, for any $t\geq0$, 
\begin{align} \label{I4W}
&|\bE[I_{4t}(W_t- W_{\underline{t}}) | \mf_{\underline{t}}] | \leq C_5  (t - \underline{t})^2 (1 + |X_{\underline{t}}|^{2(l + \alpha)+1} ).
\end{align}
Moreover,
\begin{align}\label{Jw}
\bE[I_{6t}(W_t- W_{\underline{t}}) | \mf_{\underline{t}}]=\left(\sm(X_{\underline{t}})-\sm(\widehat{X}_{\underline{t}})\right)(t-\underline{t}).
\end{align}
Therefore, combining \eqref{Diff}, \eqref{I25}, \eqref{I1W}, \eqref{I3W}, \eqref{I4W}, \eqref{Jw}, the desired result follows.
\end{proof}


\begin{Lem}\label{lm9}
Let $a \in \mathbb{R}$ and $\epsilon \in (0; +\infty)$ be fixed. There exists a positive constant $C$ which does not depend on $\Delta$, such that for any $ t \geq 0$ 
\begin{align*}
 \bE\left[ a\left| X_t -  \widehat{X}_t\right|^2|\mf_{\underline{t}}\right] &
 \leq  (a+\epsilon)|X_{ \underline{t}}- \widehat{X}_{\underline{t}}|^2+(a+\epsilon)(t-\underline{t})|\sm(X_{\underline{t}})-\sm(\widehat{X}_{\underline{t}})|^2\\
 &\quad+C\left(\dfrac{a^2}{\epsilon}+|a|\right)(t-\underline{t})^2\left(1+ |X_{\underline{t}}|^{4(l + \alpha+1)}\right)
 +C \left(\dfrac{a^2}{\epsilon}+|a|\right)\Delta^2.
 \end{align*}
\end{Lem}
\begin{proof}
We write $X_t -  \widehat{X}_t = A - B+C $ with $A= X_t - X_{\underline{t}}- \sm(X_{\underline{t}})(W_t- W_{\underline{t}})$, $B= \widehat{X}_t-\widehat{X}_{\underline{t}}-\sm(\widehat{X}_{\underline{t}})(W_t- W_{\underline{t}})$,  and $C = X_{\underline{t}}-\widehat{X}_{\underline{t}}+(\sm(X_{\underline{t}})-\sm(\widehat{X}_{\underline{t}}))(W_t- W_{\underline{t}})$. Note that $a(A-B+C)^2 \leq 2(\frac{a^2}{\epsilon} + |a|)(A^2 + B^2) + (a + \epsilon)C^2$ valid for any $a\in \mathbb{R}, \epsilon>0$. Thus, 
\begin{align}\label{aX}
a\left| X_t -  \widehat{X}_t\right|^2 &\leq (a+\epsilon)\left| X_{\underline{t}} - \widehat{X}_{\underline{t}}+(\sm(X_{\underline{t}})-\sm(\widehat{X}_{\underline{t}}))(W_t- W_{\underline{t}})\right|^2 \notag\\
&\quad+2\left(\dfrac{a^2}{\epsilon}+|a|\right)\left(\left|X_t- X_{\underline{t}}-\sm(X_{\underline{t}})(W_t- W_{\underline{t}})\right|^2+\left|\widehat{X}_t-\widehat{X}_{\underline{t}}-\sm(\widehat{X}_{\underline{t}})(W_t- W_{\underline{t}})\right|^2\right).
\end{align}
Now, \eqref{markov2} yields to
$\bE\left[2( X_{\underline{t}} - \widehat{X}_{\underline{t}})(\sm(X_{\underline{t}})-\sm(\widehat{X}_{\underline{t}}))(W_t- W_{\underline{t}})|\mf_{\underline{t}}\right]=0,$
which deduces that
\begin{align}\label{xsg}
\bE\left[| X_{\underline{t}} - \widehat{X}_{\underline{t}}+(\sm(X_{\underline{t}})-\sm(\widehat{X}_{\underline{t}}))(W_t- W_{\underline{t}})|^2|\mf_{\underline{t}}\right]& = \left| X_{\underline{t}} - \widehat{X}_{\underline{t}}\right|^2+ \bE\left[\left|(\sm(X_{\underline{t}})-\sm(\widehat{X}_{\underline{t}}))(W_t- W_{\underline{t}})\right|^2|\mf_{\underline{t}}\right] \notag\\
&=  |X_{ \underline{t}}- \widehat{X}_{\underline{t}}|^2+|\sm(X_{\underline{t}})-\sm(\widehat{X}_{\underline{t}})|^2(t-\underline{t}).
\end{align}
Consequently, combining \eqref{aX}, \eqref{xsg}  with Lemma \ref{lm4} and Lemma \ref{lm5}, we obtain the desired result. 
\end{proof}

\section*{Acknowledgment} 

This research is funded by the Vietnam Ministry of Education and Training under grant number B2024-CTT-05.

\end{document}